\documentclass{amsart}
\usepackage{amsthm,amsmath,amssymb}
\usepackage{mathrsfs}
\usepackage{color}
\usepackage{tikz}
\usepackage{enumitem}
\usepackage{amsmath, bm, bbm}
\usepackage{graphicx}
\usetikzlibrary{lindenmayersystems}

\newtheorem{theorem}{Theorem}[section]
\newtheorem{corollary}[theorem]{Corollary}
\newtheorem{lemma}[theorem]{Lemma}
\newtheorem{proposition}[theorem]{Proposition}
\theoremstyle{definition}
\newtheorem{definition}[theorem]{Definition}
\theoremstyle{remark}
\newtheorem{remark}[theorem]{Remark}
\numberwithin{equation}{section}

\title{random $\beta$-transformation on fat Sierpinski gasket}
\author{Tingyu Zhang\textsuperscript{1}, Karma Dajani\textsuperscript{2} and Wenxia Li\textsuperscript{3}}
\thanks{1,3: Department of Mathematics, Shanghai Key Laboratory of PMMP, East China Normal
University, Shanghai 200241, People's Republic of China.}
\thanks{2: Department of Mathematics, Utrecht University, Fac Wiskunde en informatica and MRI, Budapestlaan 6, P.O. Box 80.000, 3508 TA Utrecht, the Netherlands.}
\thanks{Email addresses: tingyuzhangecnu@163.com(T. Zhang), k.dajani1@uu.nl(K. Dajani), \\ wxli@math.ecnu.edu.cn(W. Li).}

\begin{document}
\maketitle

\begin{abstract}
We consider the iterated function system (IFS)
$$f_{\vec{q}}(\vec{z})=\frac{\vec{z}+\vec{q}}{\beta},\vec{q}\in\{(0,0),(1,0),(0,1)\}.$$
As is well known, for $\beta = 2$ the attractor, $S_\beta$, is a fractal called the Sierpi\'nski gasket(or sieve) and for $\beta>2$ it is also a fractal. Our goal is to study greedy, lazy and random $\beta$-transformations on the attractor for this IFS with $1<\beta<2$.
For $1<\beta\leq 3/2$, $S_\beta$ is a triangle and
it is shown that the greedy transformation $T_\beta$ and the lazy transformation $L_\beta$ are isomorphic and they both admit an absolutely continuous invariant measure.
We show that all $\beta$-expansions of a point $\vec{z}$ in $S_\beta$ can be generated by a random map $K_\beta$ defined on $\{0,1\}^\mathbb{N}\times\{0,1,2\}^\mathbb{N}\times S_\beta$
and $K_\beta$ has a unique invariant measure of maximal entropy when $1<\beta\leq\beta_*$,
where $\beta_*\approx 1.4656$ is the root of $x^3-x^2-1=0$.
We also show existence of a $K_\beta$-invariant probability measure, absolutely continuous with respect to $m_1\otimes m_2 \otimes \lambda_2$,
where $m_1, m_2$ are product measures on $\{0,1\}^\mathbb{N},\{0,1,2\}^\mathbb{N}$, respectively,
and $\lambda_2$ is the normalized Lebesgue measure on $S_\beta$.
For $3/2<\beta\leq \beta^*$, where $\beta^*\approx 1.5437$ is the root of $x^3-2x^2+2x=2$, there are radial holes in $S_\beta$.
In this case, $K_\beta$ is defined on $\{0,1\}^\mathbb{N}\times S_\beta$.
We also show that it has a unique invariant measure of maximal entropy.

~
\end{abstract}

\section{Introduction}
Let $\beta>1$ and consider the $iterated\ function\ system$(IFS):
\begin{equation*}
f_{\vec{q}_0}(\vec{z})=\frac{\vec{z}+\vec{q}_0}{\beta},f_{\vec{q}_1}(\vec{z})=\frac{\vec{z}+\vec{q}_1}{\beta},f_{\vec{q}_2}(\vec{z})=\frac{\vec{z}+\vec{q}_2}{\beta},
\end{equation*}
 where the coordinates of the three points $\vec{q}_0,\vec{q}_1,\vec{q}_2$ are $(0,0),(1,0),(0,1)$, respectively.
It is well known that there exists a unique non-empty compact set $S_\beta\subset \mathbb{R}^2$ such that
$S_\beta=\cup_{i=0}^2f_{\vec{q}_i}(S_\beta)$; see \cite{F} for further details. The $attractor$ for the IFS, $S_\beta$, is a Sierpinski gasket. Denote the convex hull of $S_\beta$ by $\Delta$ which is a triangle with vertices at $(0,0),(\frac{1}{\beta-1},0)$ and $(0,\frac{1}{\beta-1})$. For every point $\vec{z}\in S_\beta$, there exists a sequence $(a_i)_{i=1}^\infty \in\{\vec{q}_0,\vec{q}_1,\vec{q}_2\}^\mathbb{N}$ such that
$$\vec{z}=\lim_{n\rightarrow\infty}f_{a_1}\circ\cdots\circ f_{a_n}(\vec{q}_0)=\sum_{i=1}^\infty \frac{a_i}{\beta^{i}}.$$
We call $(a_i)_{i=1}^\infty $ a $coding$ of $\vec{z}$ and $\sum_{i=1}^\infty a_i\beta^{-i}$ a $representation$ of $\vec{z}$ in base $\beta$.

Let $i\in \{0,1,2\}$.
For $\beta>2$, the images $f_{\vec{q}_i}(\Delta)$ are disjoint.
In the case the IFS $\{f_{\vec{q}_i}\}$ satisfies the strong separation condition and each point in $S_\beta$ has a unique coding.
For $\beta=2$, the sets $f_{\vec{q}_i}(\Delta)$ overlap only at the vertices.
Therefore only countably many points in $S_\beta$ have two codings, and all other points have a unique coding.

When $\beta\in(1,2)$, we call $S_\beta$ a $fat\ Sierpinski\ gasket$.
For $1<\beta\leq 3/2$, we have a non-empty triple overlap $C_{012}:=f_{\vec{q}_0}(\Delta)\cap f_{\vec{q}_1}(\Delta)\cap f_{\vec{q}_1}(\Delta)$(see Figure \ref{P1}),
and $S_\beta=\Delta$. In this case Lebesgue almost every point in $S_\beta$ has a continuum of
codings (see \cite[Theorem3.5]{S2007}).
For $3/2<\beta< 2$, there are holes in $S_\beta$ as well as overlaps (see Figure \ref{radial} for example), which make its structure more complex. In \cite{BMS}, Broomhead et al. described two special types of structures: those in which holes are radially distributed and those that are totally self-similar.
We are more interested in $S_\beta$ with the former characteristics. For more results of the Hausdorff dimension of the attractors, see \cite{KL,HP,SS,JP,H}.

In this article, we shall consider the random transformation on the two-dimensional Sierpinski gasket $S_\beta$.
The main motivation of this consideration is from the random $\beta$-expansion in an interval.
In \cite{DK2003}, Dajani and Kraaikamp introduced a random $\beta$-transformation $K_\beta$ on
$\{0,1\}^\mathbb{N}\times[0,[\beta]/(\beta-1)]$ associated with the `greedy' map and the `lazy' map.
They showed that the greedy expansion is isomorphic to the lazy expansion(see \cite{DK2002}).
Dajani and de Vries\cite{DV2005} showed that $K_\beta$ has a unique measure of maximal entropy.
In \cite{DV2007} they also proved the existence and uniqueness of a $K_\beta$-invariant measure, absolutely continuous with respect to $m_p\otimes \lambda$, where $m_p$ is the Bernoulli measure on $\{0, 1\}^\mathbb{N}$ and  $\lambda$ is the normalized Lebesgue measure on $[0,[\beta]/(\beta-1)]$. Inspired by these results, we consider transformations which are defined on fat Sierpinski triangles with overlapping structures.

The rest of the article is organized as follows.
We suppose $\beta\in(1,3/2]$ in sections 3, 4 and 5.
In section 2, we focus on some basic concepts and notations needed in the rest of the paper.
In section 3, we introduce the greedy and lazy transformations on $S_\beta$ and prove that they are isomorphic.
We also prove the existence of invariant measures that are absolutely continuous with respect to Lebesgue measure (acim) for these two transformations.
In section 4, we give the definition of the random transformation $K_\beta$ on $\{0,1\}^\mathbb{N}\times\{0,1,2\}^\mathbb{N}\times S_\beta$
and prove that it has a unique invariant measure of maximal entropy
for $1<\beta\leq\beta_*$, where $\beta_*\approx 1.4656$ is the root of $x^3-x^2-1=0$.
In section 5, we give a position dependent random map $R$ on $S_\beta$.
With two skew product transformations, we establish a connection between $R$ and $K_\beta$,
and finally prove that $K_\beta$ has an invariant measure of the form $m_1\otimes m_2 \otimes \mu_\beta$,
where $m_1$ is the product measure on $\{0,1\}^\mathbb{N}$ with weights $\{p,1-p\}$,
$m_2$ is the product measure on $\{0,1,2\}^\mathbb{N}$ with weights $\{s,t,1-s-t\}$,
and $\mu_\beta$ is $R$-invariant and absolutely continuous with respct to $\lambda_2$, the normalized Lebesgue measure on $S_\beta$.
In section 6, we modify the definition of the random transformation $K_\beta$ for $3/2<\beta\leq\beta^*$ and prove that it also has a unique invariant measure of maximal entropy, where $\beta^*\approx 1.5437$ is the root of $x^3-2x^2+2x=2$.
\section{preliminary}
Given $1<\beta<2$, recall that the fat Sierpinski gasket $S_\beta$ is the self-similar set in $\mathbb{R}^2$ generated
by the IFS

\begin{equation}\label{IFS}
f_{\vec{q}_0}(\vec{z})=\frac{\vec{z}+\vec{q}_0}{\beta},f_{\vec{q}_1}(\vec{z})=\frac{\vec{z}+\vec{q}_1}{\beta},f_{\vec{q}_2}(\vec{z})=\frac{\vec{z}+\vec{q}_2}{\beta},
\end{equation}
 where $\vec{q}_0,\vec{q}_1,\vec{q}_2$ are $(0,0),(1,0),(0,1)$, respectively.
For every point $\vec{z}\in S_\beta$, there exists a sequence $(a_i)_{i=1}^\infty \in\{\vec{q}_0,\vec{q}_1,\vec{q}_2\}^\mathbb{N}$ such that
$$\vec{z}=\sum_{i=1}^\infty \frac{a_i}{\beta^{i}}.$$
Recall that $\Delta$ is the convex hull of $S_\beta$ which is an isosceles right triangle.

\medskip
Consider the ordering of points on the plane. We write $(x_1,y_1)<(x_2,y_2)$ if $x_1+y_1< x_2+y_2$, or $x_1+y_1= x_2+y_2$ and $ y_1<y_2$.
Notice that $\vec{q}_0 < \vec{q}_1< \vec{q}_2$.

\medskip
Let $ \Omega= \{0,1\}^\mathbb{N}$ with the product $\sigma$-algebra $\mathcal{A}$ and $\Upsilon=\{0,1,2\}^\mathbb{N}$ with the product $\sigma$-algebra $\mathcal{B}$.
Define metrics $d$ and $\rho$ on $\Omega$ and $\Upsilon$ respectively by
$$ d((\omega_i),(\omega'_i))=2^{-min\{k:\omega_k\neq \omega'_k\}} $$
and
$$ \rho((\upsilon_i),(\upsilon'_i))=3^{-min\{k:\upsilon_k\neq \upsilon'_k\}} .$$

Throughout the article, the lexicographical ordering on $\Omega$, $\Upsilon$ and $\{\vec{q}_0,\vec{q}_1,\vec{q}_2\}^\mathbb{N}$ are all denoted by $\prec$ and $\preceq$.
More precisely, for two sequences $(c_i), (d_i)$ we write $(c_i) \prec (d_i)$
if $c_1 < d_1$, or there exists $k \geq 2$ such that $c_i = d_i$ for all $1 \leq i < k$ and $c_k < d_k$.
Similarly, we write $(c_i) \preceq  (d_i)$ if $(c_i) \prec (d_i)$ or $(c_i) = (d_i)$.

We will use the following concepts and properties.
\begin{definition}\cite{SSM}
A topological space X is called $separable$ if it has a countable dense set.
A topological space is called $completely\ metrizable$ if its topology is
induced by a complete metric.
A $Polish\ space$ is a separable, completely metrizable topological space.
\end{definition}
There are some elementary observations.
\begin{itemize}[leftmargin=*]
\item[] (i) The real line $\mathbb{R}$ with the usual topology is Polish.
\item[] (ii) Any countable discrete space is Polish. In particular, $\{0, 1,2\}$ with the discrete topology is Polish.
\item[] (iii) The product of countably many Polish spaces is Polish. In particular,
$\Omega=\{0,1\}^\mathbb{N}$ and $\Upsilon=\{0,1,2\}^\mathbb{N}$ are Polish.
\end{itemize}
\begin{theorem}\cite{SSM}\label{polish}
Let $X, Y$ be Polish spaces, $A$ a Borel subset of $X$, and
$f : A\rightarrow Y$ a one-to-one Borel map. Then $f(A)$ is Borel.
\end{theorem}

\begin{definition}\cite{DC}
$A\ dynamical\ system$ is a quadruple $(X,\mathcal{F},\mu,Q)$,
where $X$ is a non-empty set, $\mathcal{F}$ is a $\sigma$-algebra on $X$, $\mu$ is a probability
measure on $(X,\mathcal{F}$) and $Q: X\rightarrow X$ is a surjective $\mu$-measure preserving
transformation.
\end{definition}
If $(X,\mathcal{F},\mu,Q)$ is a dynamical system, and $x \in X$, we call the
sequence
$$x,Qx, \ldots,Q^nx,\ldots$$
the $Q$-orbit of $x$.
\begin{definition}\cite{DC}
Two dynamical systems $(X,\mathcal{F},\mu,Q)$ and $(Y,\mathcal{G},\nu,U)$ are \textit{isomorphic} if there exist measurable sets $N \subset X$ and $M\subset Y$ with $\mu(N) =
0 =\nu(M)$ and $Q(X \setminus N)\subset X \setminus N, U(Y \setminus M) \subset Y \setminus M,$ and finally if there exists a measurable map $\psi: X \setminus N \rightarrow Y \setminus M$ such that (i)–(iv) are
satisfied for the systems restricted to $X \setminus N$ and $Y \setminus M$:\\
(i) $\psi$ is one-to-one and onto,\\
(ii) $\psi$ is measurable, i.e., $\psi^{-1}(G) \in \mathcal{F}$ for all $G\in\mathcal{G}$,\\
(iii) $\psi$ preserves the measures, i.e., $\nu(G) =\mu(\psi^{-1}(G))$,\\
(iv) $\psi$ preserves the dynamics of $Q$ and $U$, i.e., $\psi\circ Q = U \circ \psi$,\\
 The map $\psi$ is called an $isomorphism$.
\end{definition}
\begin{proposition}\cite{DC}
Entropy is isomorphism invariant.
\end{proposition}

\section{greedy and lazy transformation for $1<\beta\leq 3/2$ }
Let $1<\beta\leq 3/2$, then $S_\beta=\Delta$. Divide $S_\beta$ into the following sets according to the overlapping structure of $f_{\vec{q}_i}(S_\beta)$(see Figure \ref{P1}).
\begin{equation}\label{partition}
\begin{split}
&E_0=[0,\frac{1}{\beta})\times [0,\frac{1}{\beta}),\\
&E_1=\{(x,y):0 \leq y < \frac{1}{\beta},  \frac{1}{\beta(\beta-1)} < x+y  \leq  \frac{1}{\beta-1}\},\\
&E_2=\{(x,y):0 \leq x < \frac{1}{\beta},  \frac{1}{\beta(\beta-1)}<x+y \leq \frac{1}{\beta-1}\},\\
&C_{01}=\{(x,y):x \geq \frac{1}{\beta},0 \leq y < \frac{1}{\beta},   x+y\leq \frac{1}{\beta(\beta-1)}\},\\
&C_{12}=\{(x,y):x\geq \frac{1}{\beta},  y\geq \frac{1}{\beta},  \frac{1}{\beta(\beta-1)} < x+y \leq \frac{1}{\beta-1}\},\\
&C_{02}=\{(x,y):0 \leq x <  \frac{1}{\beta}, y\geq \frac{1}{\beta}, x+y \leq \frac{1}{\beta(\beta-1)}\},\\
&C_{012}=\{(x,y): x\geq \frac{1}{\beta},  y\geq \frac{1}{\beta}, x+ y \leq \frac{1}{\beta(\beta-1)}\}.
\end{split}
\end{equation}
Notice that $C_{012}$ is a single point set if $\beta=3/2$.
Let
$$C=C_{01}\cup C_{12}\cup C_{02},\ \ E=\cup_{i=0}^2E_i.$$

\begin{definition}\label{Tmap}
The greedy transformation $T_\beta$ from $S_\beta$ into $S_\beta$ is given by
\begin{equation}\label{T1}
T_\beta(\vec{z})=
\begin{cases}
\beta \vec{z}, &\mbox{if }\vec{z}\in E_0,\\
\beta \vec{z}-\vec{q}_1, &\mbox{if }\vec{z}\in C_{01}\cup E_1,\\
\beta \vec{z}-\vec{q}_2,  &\mbox{if }\vec{z}\in C_{012}\cup C_{12}\cup C_{02}\cup E_2.
\end{cases}
\end{equation}
The lazy transformation $L_\beta$ from $S_\beta$ into $S_\beta$ is given by
\begin{equation}\label{L}
L_\beta(\vec{z})=
\begin{cases}
\beta \vec{z}, &\mbox{if }\vec{z}\in C_{012}\cup C_{01}\cup C_{02}\cup E_0,\\
\beta \vec{z}-\vec{q}_1, &\mbox{if }\vec{z}\in  C_{12}\cup E_1,\\
\beta \vec{z}-\vec{q}_2,  &\mbox{if }\vec{z}\in E_2.
\end{cases}
\end{equation}
\end{definition}

\begin{figure}[h]
\begin{tikzpicture}[scale=2.6,very thick]
\pgfsetfillopacity{0.6}
\fill[fill=blue!50,draw=black,thin] (0,0) -- (64/33,0)node[below]{$\frac{1}{\beta(\beta-1)}$}-- (0,64/33)node[left]{$\frac{1}{\beta(\beta-1)}$}--cycle;
\fill[fill=blue!50,draw=black,thin] (8/11,0)node[below]{$\frac{1}{\beta}$}--(8/3,0)node[below]{$\frac{1}{\beta-1}$}--(8/11,64/33)--cycle;
\fill[fill=blue!50,draw=black,thin] (0,8/11)node[left]{$\frac{1}{\beta}$} -- (64/33,8/11)-- (0,8/3)node[left]{$\frac{1}{\beta-1}$}--cycle;
\draw [->] (-0.3,0) -- (3,0) node[at end, below] {$x$};
\draw [->] (0,-0.15) -- (0,3) node[at end, left] {$y$};
\node[below] at (1/3,1/4) {$E_0$};
\node[below] at (40/33,1/4) {$C_{01}$};
\node[below] at (24/11,1/4) {$E_1$};
\node[below] at (1/3,40/33) {$C_{02}$};
\node[below] at (40/33,40/33) {$C_{12}$};
\node[below] at (1/3,23/11) {$E_2$};
\node[below] at (29/33,31/33) {$C_{012}$};
\end{tikzpicture}
\caption{$S_\beta$ for $1<\beta\leq \frac{3}{2}$}
\label{P1}
\end{figure}
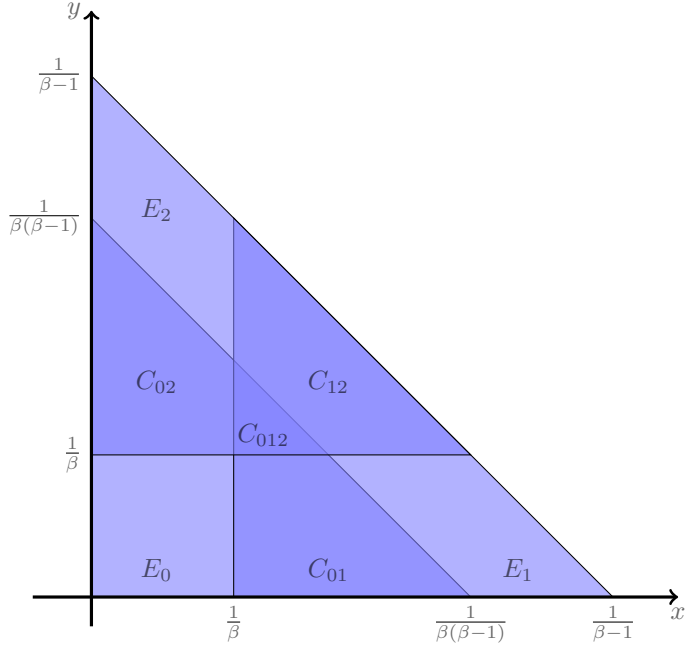
Notice that
\begin{equation*}
\begin{split}
&T_\beta|_{E_i}=L_\beta|_{E_i}=f_{\vec{q}_i}^{-1} \text{ for } i=\{0,1,2\},\\
&T_\beta|_{C_{ij}}=f_{\vec{q}_{j}}^{-1},L_\beta|_{C_{ij}}=f_{\vec{q}_i}^{-1}\text{ for } ij \in \{01,12,02\},\\
&T_\beta|_{C_{012}}=f_{\vec{q}_2}^{-1},L_\beta|_{C_{012}}=f_{\vec{q}_0}^{-1},
\end{split}
\end{equation*}
and
\begin{equation*}
\begin{split}
&f_{\vec{q}_0}(S_\beta)=E_0\cup C_{01}\cup C_{02}\cup C_{012},\\
&f_{\vec{q}_1}(S_\beta)=E_1\cup C_{01}\cup C_{12}\cup C_{012},\\
&f_{\vec{q}_2}(S_\beta)=E_2\cup C_{12}\cup C_{02}\cup C_{012}.
\end{split}
\end{equation*}
Then it is easy to verify that $T_\beta$ and $L_\beta$ are well defined.

\medskip
Denote the Borel $\sigma$-algebra on $S_\beta$ by $\mathcal{S}$. Let $\mu$ be an arbitrary $T_\beta$-invariant probability measure on $(S_\beta,\mathcal{S})$. Let $\psi: S_\beta\rightarrow S_\beta$ be given by
$$\psi(\vec{z})=\begin{bmatrix}
              1 & 0 \\
              -1 &-1 \\
    \end{bmatrix}\vec{z}+\begin{bmatrix}
               0 \\
              \frac{1}{\beta-1} \\
    \end{bmatrix}.$$
\begin{theorem}\label{iso}
For $\beta\in(1,3/2]$, the systems $(S_\beta, \mathcal{S},\mu,T_\beta)$ and $(S_\beta, \mathcal{S},\nu,L_\beta)$ are isomorphic, where $\nu=\mu\circ\psi^{-1}.$
\end{theorem}
\begin{proof}
Denote $\begin{bmatrix}
               x' \\
              y' \\
    \end{bmatrix}=\psi(\begin{bmatrix}
               x \\
              y \\
    \end{bmatrix})$. Then
$$x'=x,\ y'=\frac{1}{\beta-1}-(x+y),\ x'+y'=\frac{1}{\beta-1}-y.$$
Therefore, we can obtain
\begin{align*}
\psi(E_0)=&\psi\{(x,y):0\leq x < \frac{1}{\beta},0\leq y < \frac{1}{\beta}\}\\
      =&\{(x',y'):0\leq x' < \frac{1}{\beta},\frac{1}{\beta-1}- \frac{1}{\beta} <x'+y' \leq \frac{1}{\beta-1}\}\\
=&\{(x',y'):0\leq x' < \frac{1}{\beta},\frac{1}{\beta(\beta-1)} <x'+y' \leq \frac{1}{\beta-1}\}\\
=&E_2.
\end{align*}
By doing the same calculation, we can also obtain
$$\psi(E_1)=E_1,
\psi(C_{01})=C_{12},
\psi(C_{02})=C_{02},
\psi(C_{012})=C_{012}.$$
Notice that $\psi^{-1}=\psi$ and $\psi(S_\beta)=S_\beta$, then we can get $\psi(E_2)=E_0$ and $\psi(C_{12})=C_{01}$. Now we show $\psi\circ T_\beta=L_\beta \circ \psi$. For $(x,y)\in E_0$,
$$\psi\circ T_\beta(\begin{bmatrix}
               x \\
              y \\
    \end{bmatrix})=\psi(\begin{bmatrix}
               \beta x \\
              \beta y \\
    \end{bmatrix})=\begin{bmatrix}
               \beta x \\
         \frac{1}{\beta-1}-   \beta x -  \beta y \\
    \end{bmatrix}.$$
Since $\psi((x,y))\in E_2$, then
$$L_\beta \circ \psi(\begin{bmatrix}
               x \\
              y \\
    \end{bmatrix})=\beta \begin{bmatrix}
              x \\
            \frac{1}{\beta-1}- x-y \\
    \end{bmatrix}-\begin{bmatrix}
               0 \\
              1 \\
    \end{bmatrix}=\begin{bmatrix}
              \beta x \\
            \frac{\beta }{\beta-1}-\beta  x-\beta y -1\\
    \end{bmatrix}
=\psi\circ T_\beta(\begin{bmatrix}
               x \\
              y \\
       \end{bmatrix}).$$
Similarly, we have
\begin{align*}
\psi\circ T_\beta(\begin{bmatrix}
               x \\
              y \\
    \end{bmatrix})&=\begin{bmatrix}
               \beta x -1 \\
         \frac{\beta}{\beta-1}-   \beta x -  \beta y\\
    \end{bmatrix}=L_\beta \circ \psi(\begin{bmatrix}
               x \\
              y \\
    \end{bmatrix}),  \text{ for }(x,y)\in C_{01}\cup E_1,\\
\psi\circ T_\beta(\begin{bmatrix}
               x \\
              y \\
    \end{bmatrix})&=\begin{bmatrix}
               \beta x \\
         \frac{\beta}{\beta-1}-   \beta x -  \beta y\\
    \end{bmatrix}=L_\beta  \circ \psi(\begin{bmatrix}
               x \\
              y \\
    \end{bmatrix}), \text{ for } (x,y)\in C_{012}\cup C_{12}\cup C_{02}\cup E_2.
\end{align*} Since $\psi:S_\beta\rightarrow S_\beta$ is a bijection, then it follows from $\nu=\mu\circ\psi^{-1}$ that $\psi$ is an isomorphism.
\end{proof}

We end this section by recalling some definitions and results in \cite{GB} due to Boyarsky, A. and G\'ora, P., but rephrased to our setting.
Let $S$ be a bounded region in $\mathbb{R}^N$($N\geq 2$ is an integer) and let $\mathcal{P}=\{S_1,S_2,\ldots,S_m\}$ be a partition of $S$, where $m>0$ is an integer.
 Let $\tau$ be a transformation from $S$ into $S$. We say $\tau$ is $piecewise\ C^2\ and\ expanding$ with respect to $\mathcal{P}$ if:

\smallskip
(a) each $S_i$ is a bounded closed domain having a piecewise $C^2$ boundary of finite $(N-1)$-dimensional measure;

(b) $\tau_i=\tau|_{S_i}$ is a $C^2$, 1-1 transformation from int($S_i$) onto its image and can
be extended as a $C^2$ transformation onto $S_i, i = 1, 2 ,\ldots, m$;

(c) there exists $0 < c < 1$ such that for any $i = 1, 2 , \ldots, m$,
\begin{equation*} \Vert D\tau_i^{-1}\Vert <c,
\end{equation*}
where $D\tau_i^{-1}$ is the derivative matrix of $\tau_i^{-1}$ and $\Vert \cdot \Vert$ is the euclidean matrix norm, i.e., $\Vert A\Vert=(\sum_{i=1}^m\sum_{j=1}^na_{ij}^2)^\frac{1}{2}$ for $m\times n$ order matrix $A=(a_{ij})$.

\smallskip
Assume that the faces of $\partial S_i$ meet at angles bounded uniformly away from 0. Fix $1\leq i \leq m$.
 Let $F$ denote the set of singular points of $\partial S_i$. At any $x \in F$ we construct the largest cone having a vertex at $x$ and which lies completely in $S_i$. Let $\theta (x)$ denote the angle subtended at the vertex of this cone.
Then define
\begin{equation}\label{b}
\gamma(S_i)=\min_{x\in F}\theta(x).
\end{equation}
Since the faces of $\partial S_i$ meet at angles bounded uniformly away from 0, $\gamma(S_i)>0.$ Let
$\alpha(S_i)=\pi/2+\gamma(S_i)$ and
\begin{equation}\label{a}
a(S_i)=|\cos(\alpha(S_i))|.
\end{equation}
Assume $a=\min\{a(S_i):i=1,\ldots,m\}$.
\begin{corollary}\label{tacim}\cite[Corollary 1]{GB}
Let $\tau: S \rightarrow S, S\subset \mathbb{R}^N$, be piecewise $C^2$ and such that some
iterate $\tau^k$ satisfies $c(1+ 1/a)< 1$($c$ and $a$ corresponds to $\tau^k$), then $\tau$ admits an absolutely continuous invariant measure (acim).
\end{corollary}

Using the above corollary, we can prove the following theorem.
\begin{theorem}\label{T}
Let $T_\beta$ be defined as in (\ref{T1}). Then $T_\beta$ admits an acim. And so does $L_\beta$ which is defined as in (\ref{L}).
\end{theorem}
\begin{proof}

Since $\beta>1$, $\lim_{n\rightarrow\infty}\beta^n=\infty$. We can choose an integer $k$ such that $$\beta^k>2(2+\sqrt{2}).$$
Use the partition $\mathcal{P}:\{S_1,S_2,S_3\}$, where
$$S_1=E_0,S_2=C_{01}\cup E_1,S_3=C_{012}\cup C_{12}\cup C_{02}\cup E_2. $$
Consider the iterate $T_\beta^k$ and the corresponding partition $\vee_{i=0}^{k-1}T_\beta^{-i}(\mathcal{P})$.
For $P_i\in \vee_{i=0}^{k-1}T_\beta^{-i}(\mathcal{P})$,
let $T^k_{\beta,i}=T_\beta^k|_{P_i}$.
Since the derivative matrix of $(T_{\beta,i}^{k})^{-1}$ is
$\begin{bmatrix}
              \frac{1}{\beta^k} & 0 \\
              0 & \frac{1}{\beta^k} \\
    \end{bmatrix}, $
then the Euclidean matrix norm, $$\Vert D(T_{\beta,i}^k)^{-1}\Vert=\frac{\sqrt{2}}{\beta^k}< \frac{2\sqrt{2}}{\beta^k}:=c<1.$$

For the partition $\vee_{i=0}^{k-1}T_\beta^{-i}(\mathcal{P})$, we have $a=\sqrt{2}/2,$
which implies
$$c(1 + \frac{1}{a})=\frac{2\sqrt{2}}{\beta^k}(1+\sqrt{2})<\frac{2\sqrt{2}}{2(2+\sqrt{2})}(1+\sqrt{2})=1.$$
By Corollary \ref{tacim}, $T_\beta$ admits an acim.

\smallskip
For $L_\beta$, we use the partition $\mathcal{Q}=\{C_{012}\cup C_{01}\cup C_{02}\cup E_0,C_{12}\cup E_1,E_2\}$. In a similar way we can also prove that $L_\beta$ has an acim.
\end{proof}

\section{Random transformation for $1<\beta\leq 3/2$}
\subsection{$K_\beta$ and some properties} From Definition \ref{Tmap}, one can see that on $E_i$ both the greedy and lazy maps are identical  which means that they assign the same digits. On $C_{ij}$, the greedy map assigns the digit $\vec{q}_j$, while the lazy map assigns the digit $\vec{q}_i$.
On $C_{012}$, the digit $\vec{q}_1$ cannot be assigned by the two maps.
Now, in order to define a random transformation, we need to randomly select the map used in the switch regions $C_{ij}$ and $C_{012}$. When $\vec{z}$ belongs to $C_{ij}$ we flip a 2-sided coin to decide which map will be applied to $\vec{z}$, while we flip a 3-sided `coin' when $\vec{z}$ belongs to $C_{012}$.

Recall $\Omega=\{0,1\}^\mathbb{N}$ and $\Upsilon=\{0,1,2\}^\mathbb{N}$.
Let $\sigma:\Omega\rightarrow\Omega$ and $\sigma':\Upsilon\rightarrow\Upsilon$ be the left shifts.
Define $ K_\beta: \Omega \times \Upsilon \times S_\beta \rightarrow  \Omega \times \Upsilon\times S_\beta$ by

\begin{equation*}
K_\beta(\omega,\upsilon,\vec{z})=
\begin{cases}
(\omega,\upsilon,\beta\vec{z}-\vec{q}_i), &\mbox{if $\vec{z}\in E_i,i=0,1,2$},\\
(\sigma\omega,\upsilon,\beta \vec{z}-\vec{q}_i), &\mbox{if $\omega_1=0$ and $\vec{z}\in C_{ij},\ ij\in\{01,12,02\}$},\\
(\sigma\omega,\upsilon,\beta \vec{z}-\vec{q}_j),  &\mbox{if $\omega_1=1$ and $\vec{z}\in C_{ij},\ ij\in\{01,12,02\}$},\\
(\omega,\sigma'\upsilon,\beta \vec{z}-\vec{q}_i),  &\mbox{if $\vec{z}\in C_{012}$ and $\upsilon_1=i\in\{0,1,2\}$}.
\end{cases}
\end{equation*}
The digits are given by
\begin{equation*}
d_1 = d_1(\omega,\upsilon,\vec{z})=
\begin{cases}
\vec{q}_i,  &\mbox{if $\vec{z}\in E_i,i=0,1,2$},\\
\   &\mbox{or $(\omega,\upsilon,\vec{z})\in \Omega\times \{\upsilon_1=i\}\times C_{012}$},\\
\  &\mbox{or $(\omega,\upsilon,\vec{z})\in \{\omega_1=0\}\times \Upsilon\times C_{ij},\ ij\in\{01,12,02\}$},\\
\vec{q}_j, &\mbox{if $(\omega,\upsilon,\vec{z})\in \{\omega_1=1\}\times \Upsilon\times C_{ij},\ ij\in\{01,12,02\}$}.
\end{cases}
\end{equation*}
Then
\begin{equation*}
K_\beta(\omega,\upsilon,\vec{z})=
\begin{cases}
(\omega,\upsilon,\beta\vec{x}-d_1), &\mbox{if $\vec{z}\in E_0\cup E_1\cup E_2$},\\
(\sigma\omega,\upsilon,\beta \vec{z}-d_1), &\mbox{if $\vec{z}\in C_{01}\cup C_{12}\cup C_{02}$},\\
(\omega,\sigma'\upsilon,\beta \vec{z}-d_1), &\mbox{if $\vec{z}\in C_{012}$}.
\end{cases}
\end{equation*}

Set $d_n=d_n(\omega,\upsilon,\vec{z})=d_1(K_\beta^{n-1}(\omega,\upsilon,\vec{z}))$, and $\pi_3:\Omega \times \Upsilon \times  S_\beta \rightarrow S_\beta$ be the canonical projection onto the third coordinate. Then

$$\pi_3(K_\beta^{n}(\omega,\upsilon,\vec{z}))=\beta^n \vec{z}-\beta^{n-1}d_1-\cdots-\beta d_{n-1}-d_n,$$
and rewriting yiels
$$\vec{z}=\frac{d_1}{\beta}+\frac{d_2}{\beta^2}+\cdots+\frac{d_n}{\beta^n}+\frac{\pi_3(K_\beta^{n}(\omega,\upsilon,\vec{z}))}{\beta^n}.$$
Since $\pi_3(K_\beta^{n}(\omega,\upsilon,\vec{z}))\in S_\beta$ and $S_\beta$ is a bounded set in $\mathbb{R}^2$, it follows that
$$\Vert\vec{z}-\sum_{i=1}^n\frac{d_i}{\beta^i}\Vert_1=\frac{\Vert\pi_3(K_\beta^{n}(\omega,\upsilon,\vec{z}))\Vert_1}{\beta^n}\rightarrow 0,$$
where $\parallel \cdot \parallel_1$ denotes the $L_1$ norm, i.e. the sum of the absolute values of the vector elements.
This shows that for all $\omega\in\Omega$ , $\upsilon\in\Upsilon$ and for all $\vec{z}\in S_\beta$ one has
$$\vec{z}=\sum_{i=1}^\infty \frac{d_i}{\beta^i}= \sum_{i=1}^\infty \frac{d_i(\omega,\upsilon,\vec{z})}{\beta^i}.$$

For each point $\vec{z}\in S_\beta$, consider the set
$$D_{\vec{z}}=\{(d_1(\omega,\upsilon,\vec{z}),d_2(\omega,\upsilon,\vec{z}),\ldots): \omega\in\Omega,\upsilon\in\Upsilon\}.$$
Recall that  $\prec$ and $\preceq$ are the lexicographical ordering on $\Omega$, $\Upsilon$ and $\{\vec{q}_0,\vec{q}_1,\vec{q}_2\}^\mathbb{N}$. The following theorem shows how the lexicographical ordering on $\Omega$ and $\Upsilon$ affect the ordering of the elements in $D_{\vec{z}}$.
\begin{theorem}\label{3.1}
Suppose $\omega,\omega'\in\Omega, \upsilon, \upsilon' \in \Upsilon$ are such that $\omega\prec\omega'$ and $\upsilon \prec \upsilon'$. Then for $\vec{z}\in S_\beta$,
$$(d_1(\omega,\upsilon,\vec{z}),d_2(\omega,\upsilon,\vec{z}),\ldots)\preceq (d_1(\omega',\upsilon',\vec{z}),d_2(\omega',\upsilon',\vec{z}),\ldots).$$
\end{theorem}
\begin{proof}
Let $m:=\inf\{i:\omega_i<\omega'_i\},\ n:=\inf\{i:\upsilon_i<\upsilon'_i\}$. Then we have $\omega_m<\omega'_m$ and $\upsilon_n<\upsilon'_n$.
Denote by $t_i$ the time of the $i^{\text{th}}$ visit to the region $\Omega\times\Upsilon\times C$ of the orbit of $(\omega,\upsilon,\vec{z})$ under $K_\beta$.
Denote by $s_j$ the time of the $j^{\text{th}}$ visit to the region $\Omega\times\Upsilon\times C_{012}$ of the orbit of $(\omega,\upsilon,\vec{z})$ under $K_\beta$. One can see that $K_\beta^{t_i}(\omega,\upsilon,\vec{z})$ hits $ \Omega\times\Upsilon\times C$ for the  $i^{\text{th}}$ time and $K_\beta^{s_j}(\omega,\upsilon,\vec{z})$ hits $ \Omega\times\Upsilon\times C_{012}$ for the  $j^{\text{th}}$ time.

Let $l=\min\{t_m,s_n\}$.
Then $\pi_3(K_\beta^{i}(\omega,\upsilon,\vec{z}))=\pi_3(K_\beta^{i}(\omega',\upsilon',\vec{z}))$ for $i=0,\ldots,l$. It follows that $d_i(\omega,\upsilon,\vec{z})=d_i(\omega',\upsilon',\vec{z})$ for $i=0,\ldots,l$.

If $l=\infty$, then $d_{i}(\omega,\upsilon,\vec{z})=d_{i}(\omega',\upsilon',\vec{z})$ for all $i$. If $l<+\infty$, then $K_\beta^{l}(\omega,\upsilon,\vec{z})=K_\beta^{l}(\omega',\upsilon',\vec{z})$ hits $ \Omega\times\Upsilon\times C$ for the  $m^{\text{th}}$ time or $ \Omega\times\Upsilon\times C_{012}$ for the  $n^{\text{th}}$ time.
Since $\omega_m<\omega'_m$ and $\upsilon_n<\upsilon'_n$, then $$d_{l+1}(\omega,\upsilon,\vec{z})=d_{1}(K_\beta^{l}(\omega,\upsilon,\vec{z}))< d_1(K_\beta^{l}(\omega',\upsilon',\vec{z}))=d_{l+1}(\omega',\upsilon',\vec{z}).$$
\end{proof}

Now we show that any representation of $\vec{z}$ can be generated from the map $K_\beta$ by choosing appropriate $\omega\in\Omega$ and $\upsilon\in\Upsilon$. We need the following lemma.
\begin{lemma}\label{delta1}
Let $\beta\in(1,3/2]$. Let $(x,y)\in S_\beta$ and $(x,y)=\sum_{i=1}^\infty a_i\beta^{-i}$ with $a_i\in\{\vec{q}_0,\vec{q}_1,\vec{q}_2\}$ be a representation of $(x,y)$ in base $\beta$. One has,
\begin{itemize}[leftmargin=*]
\item[](i) If $(x,y)\in E_i$ for some $i\in\{0,1,2\}$, then $a_1=\vec{q}_i$;
\item[](ii) If $(x,y)\in C_{ij}$ for some $ij \in \{01,12,02\}$, then $a_1\in\{\vec{q}_i,\vec{q}_j\}$;
\item[](iii) If $(x,y)\in C_{012}$, then $a_1\in\{\vec{q}_0,\vec{q}_1,\vec{q}_2\}$.
\end{itemize}
\end{lemma}
\begin{proof}
(i) Suppose $a_1\neq \vec{q}_0$. From $a_1\in\{\vec{q}_1,\vec{q}_2\}$ and $(x,y)=\sum_{i=1}^\infty a_i\beta^{-i}$ we have
$$x\geq\frac{1}{\beta}\ \text{or}\ y\geq\frac{1}{\beta}.$$
Then $(x,y)\notin E_0$.

Suppose $a_1\neq \vec{q}_1$. If $a_1=\vec{q}_0$, then
$$x+y=\sum_{i=2}^\infty\frac{\parallel a_i\parallel_1}{\beta^i}\leq\frac{1}{\beta(\beta-1)}.$$
 If $a_1=\vec{q}_2$, then  $y\geq\frac{1}{\beta}.$ In both cases, $(x,y)\notin E_1$.

Suppose $a_1\neq \vec{q}_2$. By a similar proof we have $(x,y)\notin E_2$.

(ii) Suppose $a_1\notin\{\vec{q}_i,\vec{q}_j\}$.
If $a_1=\vec{q}_2$, then $y\geq\frac{1}{\beta}$, which implies $(x,y)\notin C_{01}$.
If $a_1=\vec{q}_0$, then $x+y\leq\frac{1}{\beta(\beta-1)}$, which implies $(x,y)\notin C_{12}$.
If $a_1=\vec{q}_1$, then $x\geq\frac{1}{\beta}$, which implies $(x,y)\notin C_{02}$.

(iii) holds trivially.
\end{proof}
\begin{theorem}\label{thm w}
For $\beta\in(1,3/2]$, let $\vec{z}\in S_\beta$ and $\vec{z}=\sum_{i=1}^\infty a_i \beta^{-i}$ with $a_i\in\{\vec{q}_0,\vec{q}_1,\vec{q}_2\}$ be a representation of $\vec{z}$ in base $\beta.$ Then there exists an  $\omega\in\Omega$ and an $\upsilon\in\Upsilon$ such that $a_i=d_i(\omega,\upsilon,\vec{z})$.
\end{theorem}

\begin{proof} Let $\vec{z}_n=\sum_{i=1}^\infty\frac{a_{i+n-1}}{\beta^i}.$ Notice that $\vec{z}_1=\vec{z}.$
Denote by $k_n(\vec{z})$ and $l_n(\vec{z})$ the number of times we flip a 2-sided coin and a 3-sided coin, respectively, which are
$$k_n(\vec{z})=\sum_{i=1}^n\textbf{1}_C(\vec{z}_n),\ \ l_n(\vec{z})=\sum_{i=1}^n\textbf{1}_{C_{012}}(\vec{z}_n).$$
\begin{itemize}[leftmargin=*]
\item[-]  If $\vec{z}\in E_i$, then $k_1(\vec{z})=l_1(\vec{z})=0$. It follows from Lemma \ref{delta1} that $a_1=\vec{q}_i$. Let $\Omega_1=\Omega,\Upsilon_1=\Upsilon.$
\item[-]  If $\vec{z}\in C_{ij},ij \in \{01,12,02\}$, then $k_1(\vec{z})=1,\ l_1(\vec{z})=0$. Let $\Upsilon_1=\Upsilon.$
It follows from Lemma \ref{delta1} that $a_1\in\{\vec{q}_i,\vec{q}_j\}$.
\begin{itemize}
\item If $a_1=\vec{q}_i$, we let $\Omega_1=\{\omega\in\Omega:\omega_1=0\}$.
\item If $a_1=\vec{q}_j$, we let $\Omega_1=\{\omega\in\Omega:\omega_1=1\}$.
\end{itemize}
\item[-]  If $\vec{z}\in C_{012}$, then $k_1(\vec{z})=0,l_1(\vec{z})=1$ and by Lemma \ref{delta1}, $a_1=\vec{q}_i,\ i\in\{0,1,2\}$. We let $\Omega_1=\Omega,\ \ \Upsilon_1=\{\upsilon\in\Upsilon:\upsilon_1=i\}.$
\end{itemize}
Obviously $\Omega_1$ and $\Upsilon_1$ are cylinders of length $k_1(\vec{z})$ and $l_1(\vec{z})$ respectively, and $d_1(\omega,\upsilon,\vec{z}) = a_1$ for all $\omega\in\Omega_1,\ \upsilon\in\Upsilon_1$. By a cylinder of length $0$ we mean the whole space.

Now suppose we have obtained cylinders $\Omega_n\subseteq\cdots\subseteq\Omega_1\subseteq\Omega$ and $\Upsilon_n\subseteq\cdots\subseteq\Upsilon_1\subseteq\Upsilon$.
$\Omega_n$ and $\Upsilon_n$ are cylinders of length $k_n(\vec{z})$ and $l_n(\vec{z})$ respectively.
For all $\omega\in\Omega_n,\ \upsilon\in\Upsilon_n$, 	we have $d_1(\omega,\upsilon,\vec{z}) = a_1,\cdots, d_n(\omega,\upsilon,\vec{z}) = a_n$ and $\vec{z}_{n+1}= \pi_3(K_\beta^n(\omega,\upsilon,\vec{z})$.

\begin{itemize}[leftmargin=*]
\item[-] If $\vec{z}_{n+1}\in E_i$, then $k_{n+1}(\vec{z})=k_n(\vec{z}),\ l_{n+1}(\vec{z})=l_n(\vec{z})$ and by Lemma \ref{delta1}, $a_{n+1}=\vec{q}_i$. We let $\Omega_{n+1}=\Omega_n,\Upsilon_{n+1}=\Upsilon_n.$
\item[-] If $\vec{z}_{n+1}\in C_{ij},ij \in \{01,12,02\}$, then $k_{n+1}(\vec{z})=k_n(\vec{z})+1,\ l_{n+1}(\vec{z})=l_n(\vec{z})$. We let $\Upsilon_{n+1}=\Upsilon_n.$
 By Lemma \ref{delta1}, $a_{n+1}\in\{\vec{q}_i,\vec{q}_j\}$.
\begin{itemize}
\item If $a_{n+1}=\vec{q}_i$, we let $\Omega_{n+1}=\{\omega\in\Omega_n:\omega_{k_{n+1}}=0\}$.
\item If $a_{n+1}=\vec{q}_j$, we let $\Omega_{n+1}=\{\omega\in\Omega_n:\omega_{k_{n+1}}=1\}$.
\end{itemize}
\item[-] If $\vec{z}_{n+1}\in C_{012}$, then $k_{n+1}(\vec{z})=k_n(\vec{z}),\ l_{n+1}(\vec{z})=l_n(\vec{z})+1$. By Lemma \ref{delta1}, $a_{n+1}=\vec{q}_i,\ i\in\{0,1,2\}$. Let $\Omega_{n+1}=\Omega_n$ and $\Upsilon_{n+1}=\{\upsilon\in\Upsilon_n:\upsilon_{l_{n+1}}=i\}.$
\end{itemize}
We can see that $\Omega_{n+1}$ is a cylinder of length  $k_{n+1}(\vec{z})$ and $\Upsilon_{n+1}$ is a cylinder of length  $l_{n+1}(\vec{z})$. For all $\omega\in\Omega_{n+1},\ \upsilon\in\Upsilon_{n+1}$,
 $$d_1(\omega,\upsilon,\vec{z}) = a_1,\cdots, d_{n+1}(\omega,\upsilon,\vec{z}) = a_{n+1}.$$

If the map $K_\beta$ hits $\Omega\times\Upsilon \times C$ infinitely many times, then $k_n(\vec{z})\rightarrow \infty$ and $\cap\Omega_n$ consists of a single point.
If this happens only finitely many times, then the set $\{k_n(\vec{z}):n\in\mathbb{N}\}$ is finite and $\cap\Omega_n$ is a cylinder set. For $\cap\Upsilon_n$, it has the same property according to how many times $K_\beta$ hits $\Omega\times\Upsilon \times C_{012}$.

In all cases $\cap\Omega_n$ and $\cap\Upsilon_n$ are non-empty, and for $\omega\in \cap\Omega_n, \upsilon\in\cap\Upsilon_n$, we have $d_j(\omega,\upsilon,\vec{z})=a_j$ for all $j\geq 1$.
\end{proof}

\subsection{Unique $K_\beta$-invariant measure of maximal entropy for $\beta\in(1,\beta_*)$}
Equip $\Upsilon$ with the uniform product measure $\mathbb{P}$ and recall that $\sigma'$ is the left shift on $\Upsilon$. On the set $\Omega\times\Upsilon\times S_\beta$ we consider the product $\sigma$-algebra $\mathcal{A\times B \times S}$. Define the function
$\rho_1: \Omega \times \Upsilon\times S_\beta\rightarrow \{\vec{q}_0,\vec{q}_1,\vec{q}_2\}^\mathbb{N}$ by
$$\rho_1(\omega,\upsilon,\vec{z})=(d_1(\omega,\upsilon,\vec{z}), d_2(\omega,\upsilon,\vec{z}),\ldots).$$
Define the function
$\rho_2: \{\vec{q}_0,\vec{q}_1,\vec{q}_2\}^\mathbb{N} \rightarrow \Upsilon$ by
$$\rho_2(\vec{q}_{b_1},\vec{q}_{b_2},\vec{q}_{b_3},\ldots)=(b_1,b_2,b_3,\ldots).$$
Denote by $\varphi=\rho_2\circ\rho_1$ which is the function from $\Omega \times\Upsilon\times S_\beta$ to $\Upsilon$. Then $\varphi\circ K_\beta=\sigma'\circ \varphi$, and $\varphi$ is surjective from Theorem \ref{thm w}.

It is easily seen that $\varphi$ is measurable. In fact, the inverse image of the cylinder set with the first digit fixed is measurable in $\Omega \times\Upsilon\times S_\beta$:
\begin{align*}
\varphi^{-1}(\{(b_1,b_2,\ldots)\in \Upsilon: b_1=0\})=&(\Omega\times\Upsilon\times E_0)\cup(\{\omega\in\Omega:\omega_1=0\}\times\Upsilon\times (C_{01}\cup C_{02}))\\
&\cup(\Omega\times\{\upsilon\in\Upsilon:\upsilon_1=0\}\times C_{012}),\\
\varphi^{-1}(\{(b_1,b_2,\ldots)\in \Upsilon: b_1=1\})=&(\Omega\times\Upsilon\times E_1)\cup(\{\omega\in\Omega:\omega_1=1\}\times\Upsilon\times C_{01})\\
&\cup (\{\omega\in\Omega:\omega_1=0\}\times\Upsilon\times C_{12})\cup(\Omega\times\{\upsilon\in\Upsilon:\upsilon_1=1\}\times C_{012}),\\
\varphi^{-1}(\{(b_1,b_2,\ldots)\in \Upsilon: b_1=2\})=&(\Omega\times\Upsilon\times E_2)\cup(\{\omega\in\Omega:\omega_1=1\}\times\Upsilon\times (C_{02}\cup C_{12}))\\
&\cup(\Omega\times\{\upsilon\in\Upsilon:\upsilon_1=2\}\times C_{012}).
\end{align*}

To show that $\varphi$ is an isomorphism, let
\begin{align*}
Z_1&=\{(\omega,\upsilon,\vec{z})\in\Omega\times \Upsilon\times S_\beta: K_\beta^n(\omega,\upsilon,\vec{z})\in\Omega\times \Upsilon\times C \text{ infinitely often}\},\\
Z_2&=\{(\omega,\upsilon,\vec{z})\in\Omega\times \Upsilon\times S_\beta: K_\beta^n(\omega,\upsilon,\vec{z})\in\Omega\times \Upsilon\times C_{012} \text{ infinitely often}\},\\
D_1&=\{(b_1,b_2,\ldots)\in \Upsilon: \sum_{i=1}^\infty\frac{\vec{q}_{b_{j+i-1}}}{\beta^i}\in C \text{ for infinitely many } j\text{'s}\},\\
D_2&=\{(b_1,b_2,\ldots)\in \Upsilon: \sum_{i=1}^\infty\frac{\vec{q}_{b_{j+i-1}}}{\beta^i}\in C_{012} \text{ for infinitely many } j\text{'s}\}.
\end{align*}
Notice that
$$ Z_1=\cap_{n=1}^\infty \cup_{m=n}^\infty K_\beta^{-m}(\Omega\times \Upsilon \times C)$$
and
$$ Z_2=\cap_{n=1}^\infty \cup_{m=n}^\infty K_\beta^{-m}(\Omega\times \Upsilon \times C_{012}),$$
which imply that $Z_1$ and $Z_2$ are Borel sets in $\Omega\times \Upsilon \times S_\beta$.
Let $Z=Z_1\cap Z_2,\ D=D_1\cap D_2$,
then we have $K_\beta^{-1}(Z)=Z,\ (\sigma')^{-1}(D)=D$ and $\varphi(Z)=D$. Let $\varphi'=\varphi|_Z$.

\begin{lemma}\label{bi'}
The map $\varphi':Z\rightarrow D$ is a bimeasurable bijection.
\end{lemma}
\begin{proof}
For any sequence $(b_1,b_2,\ldots)\in D $, we can obtain a point $\vec{z}=\sum_{i=1}^\infty \vec{q}_{b_i}\beta^{-i}$.
To determine $\omega$ and $\upsilon$, we could define
\begin{align*}
&r_1=\min\{j\geq 1:\sum_{i=1}^\infty\frac{\vec{q}_{b_{j+i-1}}}{\beta^i}\in C\},\ \
r_k=\min\{j>r_{k-1}:\sum_{i=1}^\infty\frac{\vec{q}_{b_{j+i-1}}}{\beta^i}\in C\},\\
&s_1=\min\{j\geq 1:\sum_{i=1}^\infty\frac{\vec{q}_{b_{j+i-1}}}{\beta^i}\in C_{012}\},\ \
s_k=\min\{j>s_{k-1}:\sum_{i=1}^\infty\frac{\vec{q}_{b_{j+i-1}}}{\beta^i}\in C_{012}\}.
\end{align*}

\begin{itemize}[leftmargin=*]
\item[-] If $\sum_{i=1}^\infty\vec{q}_{b_{r_k+i-1}}\beta^{-i} \in C_{ij}, ij\in\{01,12,02\}$, then
$b_{r_k}\in\{i,j\}$ by Lemma \ref{delta1}.
\begin{itemize}
\item If $b_{r_k}=i$, we let $\omega_k=0$.
\item If $b_{r_k}=j$, we let $\omega_k=1$.
\end{itemize}
\item[-] If $\sum_{i=1}^\infty\vec{q}_{b_{s_k+i-1}}\beta^{-i} \in  C_{012}$, then $b_{s_k}\in\{0,1,2\}$ by Lemma \ref{delta1}. If $b_{s_k}=i$, we let $\upsilon_k=i, i\in\{0,1,2\}$.
\end{itemize}
Notice that for any $N>0$, there exsit $r,s>N$ such that
$$\sum_{i=1}^\infty \vec{q}_{b_{r+i-1}}\beta^{-i}\in C \text{ and }\sum_{i=1}^\infty \vec{q}_{b_{s+i-1}}\beta^{-i}\in C_{012},$$
which implies that $K_\beta^n(\omega,\upsilon,\vec{z})$ hits both $\Omega\times \Upsilon\times C$ and $\Omega\times \Upsilon\times C_{012}$ infinitely often.
Then the infinite sequences $\omega=(\omega_1,\omega_2,\omega_3,\ldots)\in\Omega$ and $\upsilon=(\upsilon_1,\upsilon_2,\upsilon_3,\ldots)\in\Upsilon$ can be uniquely determined.
Therefore, we can define the inverse of $\varphi'$.
Let $(\varphi')^{-1}:D\rightarrow Z$ be
$$(\varphi')^{-1}((b_1,b_2,\ldots))=(\omega,\upsilon,\sum_{i=1}^\infty\frac{\vec{q}_{b_i}}{\beta^i}).$$
If $(\omega,\upsilon,\vec{z})=(\omega',\upsilon',\vec{z'})$ then $\varphi'(\omega,\upsilon,\vec{z})=\varphi'(\omega',\upsilon',\vec{z'})$.
Since $Z=Z_1\cap Z_2$ is a Borel set in $\Omega\times\Upsilon\times S_\beta$, then we have that $(\varphi')^{-1}$ is measurable by Theorem \ref{polish}.
Hence $\varphi'$ is a bimeasurable bijection.
\end{proof}

Let $\beta_*\approx 1.4656$ be the root of $x^3-x^2-1=0$.
\begin{lemma}\label{m}
If $1<\beta<\beta_*$, then $\mathbb{P}(D)=1$.
\end{lemma}
\begin{proof}
Let us first prove $\mathbb{P}(D_2)=1$.
For any sequence $(b_1,b_2,\ldots)\in \Upsilon$ and $n\geq 4$, define
$$\vec{z}_{n}=(x_{n},y_{n})=\frac{\vec{q}_1}{\beta}+\frac{\vec{q}_2}{\beta^2}+\frac{\vec{q}_2}{\beta^3}+\frac{\vec{q}_0}{\beta^4}+\cdots+\frac{\vec{q}_0}{\beta^n}+\frac{\vec{q}_{b_1}}{\beta^{n+1}}+\frac{\vec{q}_{b_2}}{\beta^{n+2}}+\cdots.$$
Then we have $x_n\geq\frac{1}{\beta}$, and
\begin{align*}
y_n&\geq \frac{1}{\beta^{2}}+\frac{1}{\beta^{3}},\\
x_n+y_n &\leq \sum_{i=1}^3\frac{1}{\beta^{i}}+\sum_{i=1}^\infty\frac{1}{\beta^{n+i}}.
\end{align*}
If $1<\beta<\beta_*$, then
$$
\frac{1}{\beta^{2}}+\frac{1}{\beta^{3}}\geq \frac{1}{\beta},$$
and
$$\frac{1}{\beta}+\frac{1}{\beta^{2}}+\frac{1}{\beta^{3}}<\frac{1}{\beta(\beta-1)}.$$
 Since $\sum_{i=1}^\infty\frac{1}{\beta^{n+i}}=\frac{1}{\beta^n(\beta-1)}\rightarrow 0$ as $n\rightarrow \infty$, there exists an integer $N>3$ such that for any $n\geq N$,
\begin{align*}
y_n&\geq \frac{1}{\beta},\\
x_n+y_n &\leq \frac{1}{\beta(\beta-1)}.
\end{align*}
It follows that $(x_n,y_n)\in C_{012}$ for any $n\geq N$.

Let $$\boldsymbol{b}=122\underbrace{00\ldots0}_{N-3 \text{ times}}.$$
Let
\begin{align*}
D'&=\{(b_1,b_2,\ldots)\in \Upsilon: b_jb_{j+1}\ldots b_{j+N-1}=\boldsymbol{b} \text{ for infinitely many } j\text{'s}\}
\end{align*}
then $D'\subset D_2$. Now we show that $\mathbb{P}(D')=1$.
Notice that $D'=\cap_{n=1}^\infty\cup_{m=n}^\infty D_m,$ where $D_m=\{(b_1,b_2,\ldots)\in\Upsilon:b_mb_{m+1}\ldots b_{m+N-1}=\boldsymbol{b}\}.$ Let $B_n=\cup_{m=n}^\infty D_m$.
If $(b_1,b_2,\ldots)\in\Upsilon\setminus B_n$, then we have that for any $j\geq n$,
$b_jb_{j+1}\ldots b_{j+N-1}\neq \boldsymbol{b}.$ Clearly,
$$\Upsilon\setminus B_n\subseteq B':=\{(b_1,b_2,\ldots): b_{n+kN}\ldots b_{n+(k+1)N-1}\neq \boldsymbol{b}, k=0,1,2,\ldots\}.$$
Since $\mathbb{P}(\Upsilon\setminus B_n)\leq\mathbb{P}(B')=\lim_{k\rightarrow\infty}(1-1/3^N)^k=0$, then $\mathbb{P}(B_n)=1$.
It follows from $D'=\cap_{n=1}^\infty B_n$ and $B_1\supseteq B_2\supseteq\cdots$ that $\mathbb{P}(D') =\lim_{n\rightarrow\infty}B_n=1$. Then we get $\mathbb{P}(D_2)=1.$

Similarly, define
$$\vec{z}_l=(x_l,y_l)=\frac{\vec{q}_1}{\beta}+\frac{\vec{q}_0}{\beta^2}+\cdots+\frac{\vec{q}_0}{\beta^l}+\frac{\vec{q}_{b_1}}{\beta^{l+1}}+\frac{\vec{q}_{b_2}}{\beta^{l+2}}+\cdots.$$
Then we have $ x_l\geq \frac{1}{\beta}$, and
\begin{align*}
0 &\leq y_l \leq \sum_{i=1}^\infty\frac{1}{\beta^{l+i}}=\frac{1}{\beta^l(\beta-1)},\\
\frac{1}{\beta}&\leq x_l+y_l \leq \frac{1}{\beta}+\sum_{i=1}^\infty\frac{1}{\beta^{l+i}}=\frac{1}{\beta}+\frac{1}{\beta^l(\beta-1)}.
\end{align*}
Since $\lim_{l\rightarrow \infty}\frac{1}{\beta^l(\beta-1)}=0$, then there exists $L>0$ such that for any $l\geq L$,
\begin{equation*}\label{yz}
0 \leq y_l < \frac{1}{\beta},  x_l+y_l \leq \frac{1}{\beta(\beta-1)}.
\end{equation*}
It follows that $(x_l,y_l)\in C_{01}$ for any $l\geq L$. Let
\begin{equation}\label{PD''}
D''=\{(b_1,b_2,\ldots)\in \Upsilon: b_jb_{j+1}\ldots b_{j+L-1}=1\underbrace{00\ldots 0}_{L-1 \text{ times }} \text{ for infinitely many } j\text{'s}\},
\end{equation}
then $D''\subset D_1$. Since $\mathbb{P}(D'')=1$, we have $\mathbb{P}(D_1)=1$.
Therefore, $\mathbb{P}(D)=1$.
\end{proof}

The following theorem can be obtained from Lemmas \ref{bi'} and \ref{m}.
\begin{theorem}\label{4.6}
For $\beta\in(1, \beta_*]$, the dynamical systems
$(\Omega\times\Upsilon\times S_\beta, \mathcal{A\times B \times S},\nu_\beta, K_\beta)$ and
$(\Upsilon,\mathcal{B},\mathbb{P},\sigma')$ are isomorphic, where $\nu_\beta (A)=\mathbb{P}(\varphi(Z \cap A))$.
\end{theorem}
\begin{remark}\label{remark}
(i) Notice that if we change $\mathbb{P}$ to the general product measure $m_2$ on $\Upsilon$,  Lemma \ref{m} and Theorem \ref{4.6} remain true.\\
(ii) Since $\mathbb{P}$ is the unique measure of maximal entropy on $\Upsilon$, the above theorem implies that any other $K_\beta$-invariant measure with support $Z$ has entropy strictly less than $\log 3$.
We now investigate the entropy of $K_\beta$-invariant measure $\mu$ for which $\mu(Z^c)>0$.
\end{remark}

Divide $Z^c$ into three Borel sets as follows:
\begin{align*}
Z^c&=(Z_1\cap Z_2)^c\\
   &=Z_1^c \cup Z_2^c\\
   &=(Z_1^c \setminus Z_2^c)\cup(Z_2^c \setminus Z_1^c)\cup(Z_1^c \cap Z_2^c)\\
   &:=Z_3 \cup  Z_4 \cup Z_5,
\end{align*}
where
\begin{align*}
Z_3=&(Z_1^c \setminus Z_2^c)\\
=&\{(\omega,\upsilon,\vec{z})\in\Omega\times \Upsilon\times S_\beta: K_\beta^n(\omega,\upsilon,\vec{z})\in\Omega\times \Upsilon\times C \text{ for finitely many } n\text{'s}\\
&\text{and } K_\beta^n(\omega,\upsilon,\vec{z})\in\Omega\times \Upsilon\times C_{012} \text{ infinitely often}\},\\
Z_4=&(Z_2^c \setminus Z_1^c)\\
=&\{(\omega,\upsilon,\vec{z})\in\Omega\times \Upsilon\times S_\beta: K_\beta^n(\omega,\upsilon,\vec{z})\in\Omega\times \Upsilon\times C_{012} \text{ for finitely many } n\text{'s},\\
&\text{and } K_\beta^n(\omega,\upsilon,\vec{z})\in\Omega\times \Upsilon\times C \text{ infinitely often}\},\\
Z_5=&(Z_1^c \cap Z_2^c)\\
=&\{(\omega,\upsilon,\vec{z})\in\Omega\times \Upsilon\times S_\beta: K_\beta^n(\omega,\upsilon,\vec{z}) \in\Omega\times \Upsilon\times C \text{ for finitely many } n\text{'s}, \\
&\text{ and } K_\beta^n(\omega,\upsilon,\vec{z})\in\Omega\times \Upsilon\times C_{012} \text{ for finitely many } n\text{'s}\}.
\end{align*}

We first prove the following lemma.
\begin{lemma}\label{entropy3}
Let $\beta\in(1, \beta_*)$. Let $\mu_3$ be a $K_\beta$-invariant measure for which $\mu_3(Z_3)=1$. Then $h_{\mu_3}(K_\beta)< \log 3$. Similarly, let $\mu_4$ and $\mu_5$ be $K_\beta$-invariant measures for which $\mu_4(Z_4)=\mu_5(Z_5)=1$. Then $h_{\mu_4}(K_\beta),h_{\mu_5}(K_\beta)< \log 3$ also holds.
\end{lemma}
\begin{proof}
Let
\begin{align*}
H_3=\{\vec{z}=\sum_{i=1}^\infty\frac{\vec{q}_{b_i}}{\beta^i} \in S_\beta: &\sum_{i=1}^\infty\frac{\vec{q}_{b_{j+i-1}}}{\beta^i} \text{ belongs to } C_{012} \text{ for infinitely many } j\text{'s,}\\
&\text{ and never belongs to }C \}.
\end{align*}
Then $\Omega\times\Upsilon\times H_3 \subseteq K_\beta^{-1}(\Omega\times\Upsilon\times H_3)$ and
$\cup_{i=0}^\infty K_\beta^{-i}(\Omega\times\Upsilon\times H_3)=Z_3$.
It follows that
$\mu_3(Z_3)=\lim_{i\rightarrow \infty}\mu_3 (K_\beta^{-i}(\Omega\times\Upsilon\times H_3))=1.$
Since  $\mu_3$ is $K_\beta$-invariant, then
$$\mu_3(\Omega\times\Upsilon\times H_3)=\mu_3 (K_\beta^{-1}(\Omega\times\Upsilon\times H_3))=\mu_3 (K_\beta^{-2}(\Omega\times\Upsilon\times H_3))=\cdots=1.$$
Thus it is enough to study the entropy with respect to $\mu_3$ of the map $K_\beta$ restricted to $\Omega\times\Upsilon\times H_3$.
Let $\pi_1,\pi_2,\pi_3$ be the canonical projection onto the three coordinates respectively.
Notice that the action of the transformation $K_\beta$ on the first coordinate is an identity,
which implies that $K_\beta$ is essentially a product transformation $I_\Omega\times K_\beta' $,
where $K_\beta'=(\pi_2\circ K_\beta) \times (\pi_3\circ K_\beta)$ on $\Upsilon\times H_3$ and  $I_\Omega$ is the identity on $\Omega$.
Since $(\upsilon,\vec z)\in\Upsilon\times S_\beta$ and $\omega\in \Omega$ are independent,
it follows from $h_\mu(I_\Omega)=0$ for any measure $\mu$ on $(\Omega, \mathcal{A})$ and that
$$h_{\mu_3}(K_\beta)=h_{\mu_3'}(K_\beta'),$$
where $\mu_3'(B\times H)=\mu_3(\Omega\times B \times H)$ for $B\in \mathcal{B},H\in (H_3\cap\mathcal{S})$.\\

Let
\begin{align*}
D_3=\{(b_1,b_2,\ldots)\in \Upsilon:&\sum_{i=1}^\infty\frac{\vec{q}_{b_{j+i-1}}}{\beta^i} \text{ belongs to } C_{012} \text{ for infinitely many } j\text{'s,}\\
&\text{ and never belongs to }C \}.
\end{align*}
Define a map $\phi$ from $(\Upsilon\times H_3, \mathcal{B}\times (H_3\cap\mathcal{S}), \mu_3',K_\beta')$ to $(D_3,D_3\cap\mathcal{B},\mu_3'\circ \phi^{-1},\sigma')$ as
$$\phi(\upsilon,\vec z)=\rho_2(\rho_1(0^\infty,\upsilon,\vec z)).$$
Since $\vec z \in H_3$, then $\phi$ is well defined and bijective.  $\phi$ is measurable and the inverse is also measurable by Theorem \ref{polish}. Finally, $\phi$ preserves the measure and $\phi\circ K_\beta'=\sigma'\circ \phi$. Then $\phi$ is an isomorphism and it follows that
$$h_{\mu_3}(K_\beta)=h_{\mu_3'}(K_\beta')=h_{\mu_3'\circ \phi^{-1}}(\sigma')\leq h_{\mathbb{P}}(\sigma')=\log 3.$$

Since $\mathbb{P}$ is the unique measure of maximal entropy on $D_3$, to show $h_{\mu_3}(K_\beta)<\log 3$, it is enough to prove that $\mu_3'\circ\phi^{-1}\neq\mathbb{P}$. This is done by contradiction.
If $\mu_3'\circ\phi^{-1}=\mathbb{P}$, then
$$\mathbb{P}(D_3)=\mu_3'(\Upsilon\times H_3)=\mu_3(\Omega\times\Upsilon\times H_3)=1.$$
Since $D_3\subset (D'')^c$, where $D''$ is defined as in (\ref{PD''}), then $\mathbb{P}((D'')^c)=1$, which is a contradiction to $\mathbb{P}(D'')=1$. Therefore, $$h_{\mu_3}(K_\beta)<\log 3.$$
Let
\begin{align*}
H_4=\{\vec{z}=\sum_{i=1}^\infty\frac{\vec{q}_{b_i}}{\beta^i} \in S_\beta: &\sum_{i=1}^\infty\frac{\vec{q}_{b_{j+i-1}}}{\beta^i} \text{ belongs to } C \text{ for infinitely many } j\text{'s,}\\
&\text{ and never belongs to }C_{012} \},\\
H_5=\{\vec{z}=\sum_{i=1}^\infty\frac{\vec{q}_{b_i}}{\beta^i} \in S_\beta: &\sum_{i=1}^\infty\frac{\vec{q}_{b_{j+i-1}}}{\beta^i} \text{ never belongs to } C\\
&\text{ and never belongs to }C_{012} \},\\
=\{\vec{z}\in S_\beta: \vec{z} \text{ has a u}& \text{nique } \beta\text{-expansion}\}.
\end{align*}
Then it follows that
$$\mu_4(\Omega\times\Upsilon\times H_4)=\mu_5(\Omega\times\Upsilon\times H_5)=1.$$
We can also obtain that $h_{\mu_4}(K_\beta),h_{\mu_5}(K_\beta)< \log 3$ using the similar method.
\end{proof}

From the above lemma we can obtain the upper bound of the entropy of $K_\beta$-invariant measure for which $Z^c$ has positive measure.
\begin{lemma}\label{max entropy}
Let $\beta\in(1, 3/2]$. Let $\mu$ be a  $K_\beta$-invariant measure for which $\mu(Z^c)>0$. Then $h_\mu(K_\beta) < \log 3.$
\end{lemma}
\begin{proof}
Notice that $Z,Z_3,Z_4 $ and $ Z_5$ are pairwise disjoint and the union is $\Omega \times \Upsilon\times S_\beta$.
Since $Z, Z_3,Z_4$ and $Z_5$ are $K_\beta$-invariant, then there exist $K_\beta$-invariant probability measures $\mu_{12},\mu_3,\mu_4$ and $\mu_5$ concentrated on $Z,Z_3,Z_4$ and $Z_5$, respectively, such that
$$\mu=(1-\alpha_3-\alpha_4-\alpha_5)\mu_{12}+\alpha_3\mu_3+\alpha_4\mu_4+\alpha_5\mu_5,$$
where $ 0\leq \alpha_3,\alpha_4,\alpha_5\leq 1$ and $0<\alpha_3+\alpha_4+\alpha_5\leq 1$.
Then
$$h_\mu(K_\beta)=(1-\alpha_3-\alpha_4-\alpha_5) h_{\mu_{12}}(K_\beta)+\alpha_3 h_{\mu_3}(K_\beta)+\alpha_4 h_{\mu_4}(K_\beta)+\alpha_5 h_{\mu_5}(K_\beta).$$
Since $h_{\mu_{12}}(K_\beta) \leq \log 3$ by Remark \ref{remark} and $h_{\mu_3}(K_\beta), h_{\mu_4}(K_\beta), h_{\mu_5}(K_\beta) <\log 3$ by Lemma \ref{entropy3}, then the result follows.
\end{proof}

Now we obtain the main result in this section.
\begin{theorem}
Let $\beta\in(1, \beta_*)$. The measure $\nu_\beta (A)=\mathbb{P}(\varphi(Z \cap A))$ is the unique $K_\beta$-invariant measure of maximal entropy.
\end{theorem}

\section{An absolutely continuous invariant measure for $K_\beta$ when $1<\beta\leq \frac{3}{2}$}

The analysis is similar in the two cases: $\beta\in(1,3/2)$ and $\beta=3/2$. Without special statement we assume $\beta\in(1,3/2)$ in this section.
A brief description of the case of $\beta=3/2$ is given in the remarks.
Recall the product measure $m_1$ on $\Omega$ with weights $\{p,1-p\}$ and $m_2$ on $\Upsilon$ with weights $\{s,t,1-s-t\}$.
Consider the measure space $(S_\beta,\mathcal{S},\lambda_2)$, where $\lambda_2$ is the normalized Lebesgue measure.
In this section we will prove that $K_\beta$ has an invariant measure of the form $m_1\otimes m_2 \otimes \mu_\beta$, where $\mu_\beta$ is absolutely continuous with respect to $\lambda_2$.
We will show the result by several steps.\\

\medskip
\noindent
\textbf{Step 1: a position dependent random transformation $R$.}

\smallskip
Bahsoun and G\'ora \cite{BG}, gave a sufficient condition for the existence of an absolutely continuous invariant measure for a random map with position dependent probabilities on a bounded domain of $\mathbb{R}^N$. We take some of their results a little further.

For $k=1,\ldots,K$, let $\tau_k:S_\beta\rightarrow S_\beta$ be piecewise one-to-one and $C^2$, non-singular transformations on a common partition $\mathcal{P}$ of $S_\beta : \mathcal{P} = \{S_1, \ldots, S_q \}$ and $\tau_{k,i}= \tau_k|_{S_i}, i = 1, \ldots, q.$
Let $p_k: S_\beta\rightarrow [0,1]$ be piecewise $C^1$ functions such that $ \sum_{k=1}^Kp_k=1.$
Denote by $R=\{\tau_1,\ldots,\tau_K;p_1(\vec{z}),\ldots,p_K(\vec{z})\}$ the position dependent random map, i.e., $R(\vec z)=\tau_k(\vec z)$ with probability $p_k(\vec z)$.
Define the transition function for $R$ as follows:
$$\mathbf{P}(\vec{z},A)=\sum_{k=1}^Kp_k(\vec{z})\mathbbm{1}_A(\tau_k(\vec{z})),$$
where $A$ is any measurable set and $\mathbbm{1}_A$ denotes the indicator function of the set $A$.

The iteration of $R$ is denoted by
$R^n:=\{\tau_{k_1k_2\cdots k_n};p_{k_1k_2\cdots k_n}\},k_1k_2\cdots k_n\in\{1,2,\ldots,K\}^n$, where
$$\tau_{k_1k_2\cdots k_n}(\vec{z})=\tau_{k_n}\circ\tau_{k_{n-1}}\circ \cdots\circ\tau_{k_1}(\vec{z})$$
and
$$p_{k_1k_2\cdots k_n}(\vec{z})=p_{k_n}(\tau_{k_{n-1}}\circ\cdots\circ\tau_{k_1}(\vec{z}))\cdot p_{k_{n-1}}(\tau_{k_{n-2}}\circ\cdots\circ\tau_{k_1}(\vec{z}))\cdots p_{k_1}(\vec{z}).$$
The transition function $\mathbf{P}$ induces an operator $\mathbf{P}_*$ on the set of probability measure on  $ (S_\beta, \mathcal{S})$ defined by
$$ \mathbf{P}_*\mu(A)=\int\mathbf{P}(\vec{z},A)d\mu(\vec{z})=\sum_{k=1}^K\int_{\tau_k^{-1}(A)}p_k(\vec{z})d\mu(\vec{z})=\sum_{k=1}^K\sum_{i=1}^q\int_{\tau_{k,i}^{-1}(A)}p_k(\vec{z})d\mu(\vec{z}).$$
We say that the measure $\mu$ is $R$-invariant iff $\mathbf{P}_*\mu=\mu$.

If $\mu$ has density $f$ with respect to $\lambda_2$, then $\mathbf{P}_*\mu$ has also a density which we denote by $P_Rf$, i.e.,
$$ \int_AP_Rf(\vec z)d\lambda_2(\vec z)=\sum_{k=1}^K\sum_{i=1}^q \int_{\tau_{k,i}^{-1}(A)}p_k(\vec{z})f(\vec{z})d\lambda_2(\vec z).$$
We call $P_R$ the Perron-Frobenius operator of the random map $R$ and it has very useful properties\cite{BG}:
\begin{itemize}[leftmargin=*]
\item[] (i) $P_R$ is linear;
\item[] (ii) $P_R$ is non-negative;
\item[] (iii) $P_Rf=f \iff \mu=f\cdot \lambda_2$ is $R$-invariant;
\item[] (iv) $\Vert P_Rf\Vert_1\leq\Vert f\Vert_1$, where $\Vert \cdot\Vert_1$ denotes the $L^1$ norm;
\item[] (v) $P_{R\circ T}=P_R\circ P_T.$ In particular, $P_R^N=P_{R^N}.$
\end{itemize}

Let each $S_i$ be a bounded closed domain having a piecewise $C^2$ boundary of finite $1$-dimensional measure.
Assume that the faces of $\partial S_i$ meet at angles bounded uniformly away from $0$ and the probabilities $p_k(\vec{z})$ are piecewise $C^1$ functions on the partition $\mathcal{P}$. We assume:

\medskip
\noindent
 \textbf{Condition(A):}
\begin{equation*}
\max_{1\leq i\leq q}\sum_{k=1}^Kp_k(\vec{z})\Vert D\tau_{k,i}^{-1}(\tau_{k,i}(\vec{z}))\Vert <c<1,
\end{equation*}
where $D\tau_{k,i}^{-1}(\vec{z})$ is the derivative matrix of $\tau_{k,i}^{-1}$ at $\vec{z}$.

\medskip
Using the multidimensional notion of variation \cite{G}:
$$V(f)=\int_{\mathbb{R}^N}\Vert Df\Vert d\lambda_N=\sup\left\{\int_{\mathbb{R}^N}f\text{div}(g)d\lambda_N:g=(g_1,\ldots,g_N)\in C_0^1(\mathbb{R}^N,\mathbb{R}^N)\right\},$$
where $f\in L_1(\mathbb{R}^N)$ has bounded support, $Df$ denotes the gradient of $f$ in the distributional sense, $\text{div}(g)=\nabla\cdot g=\frac{\partial g_1}{\partial x_1}+\frac{\partial g_2}{\partial x_2}+\cdots+\frac{\partial g_N}{\partial x_N}$ is the divergence operator, and $C_0^1(\mathbb{R}^N,\mathbb{R}^N)$ is the space of continuously differentiable functions from $\mathbb{R}^N$ into $\mathbb{R}^N$ having compact support. Consider the Banach space \cite[Remark 1.12]{G},
$$BV(S_\beta)=\{f\in L_1(S_\beta):V(f)<+\infty\},$$
with the norm $\Vert f\Vert_{BV}=\Vert f\Vert_{L_1}+V(f)$.

Let $a(S_i)$ be defined as in (\ref{a}) and $\gamma(S_i)$ defined as in (\ref{b}). Now we start at points $y\in F$, where the minimal angle $\gamma(S_i)$ is attained, defining $L_y$ to be central rays of the largest regular cones contained in $S_i$. Then we extend this field of segments to $C^1$ field, making $L_y$ short enough to avoid overlapping. Let $\delta(y)$ be the length of $L_y,y\in\partial S_i$. By the compactness of $\partial S_i$ we have
$$\delta(S_i)=\inf_{y\in\partial(S_i)}\delta(y)>0.$$

Let $\vec{z}$ be a point in $\partial S_i$ and $J_{k,i}$ the Jacobian of $\tau_{k|S_i}$ at $\vec{z}$.

We recall the following theorem.
\begin{theorem}\cite[Theorem 6.3]{BG}\label{VF}
If $R$ is a random map which satisfies Condition (A), then
$$V(P_Rf)\leq c(1+1/a)V(f)+(M+\frac{c}{a\delta})\Vert f\Vert_1,$$
where $a=\min\{a(S_i):i=1,\ldots,q\}>0,\delta=\min\{\delta(S_i):i=1\ldots,q\}>0,M_{k,i}=\sup_{\vec{z}\in S_i}(Dp_k(\vec{z})-\frac{DJ_{k,i}}{J_{k,i}}p_k(\vec{z}))$ and $M=\sum_{k=1}^K\max_{1\leq i\leq q}M_{k,i}.$
\end{theorem}

Now, let $R$ be a random map which is given by $\{\tau_1,\ldots,\tau_6;p_1(\vec{z}),\ldots,p_6(\vec{z})\}$ where
\begin{align*}
\tau_1(\vec{z})&=
\begin{cases}
\beta \vec{z}-\vec{q}_i, &\mbox{if }\vec{z}\in  E_i,i\in\{0,1,2\}\\
\beta \vec{z}-\vec{q}_i, &\mbox{if }\vec{z}\in  C_{ij},ij\in\{01,02,12\}\\
\beta \vec{z}, &\mbox{if }\vec{z}\in C_{012},
\end{cases}\\
\tau_2(\vec{z})&=
\begin{cases}
\beta \vec{z}-\vec{q}_i, &\mbox{if }\vec{z}\in  E_i,i\in\{0,1,2\}\\
\beta \vec{z}-\vec{q}_i, &\mbox{if }\vec{z}\in  C_{ij},ij\in\{01,02,12\}\\
\beta \vec{z}-\vec{q}_1, &\mbox{if }\vec{z}\in C_{012},
\end{cases}\\
\tau_3(\vec{z})&=
\begin{cases}
\beta \vec{z}-\vec{q}_i, &\mbox{if }\vec{z}\in  E_i,i\in\{0,1,2\}\\
\beta \vec{z}-\vec{q}_i, &\mbox{if }\vec{z}\in  C_{ij},ij\in\{01,02,12\}\\
\beta \vec{z}-\vec{q}_2, &\mbox{if }\vec{z}\in C_{012},
\end{cases}\\
\tau_4(\vec{z})&=
\begin{cases}
\beta \vec{z}-\vec{q}_i, &\mbox{if }\vec{z}\in  E_i,i\in\{0,1,2\}\\
\beta \vec{z}-\vec{q}_j, &\mbox{if }\vec{z}\in  C_{ij},ij\in\{01,02,12\}\\
\beta \vec{z}, &\mbox{if }\vec{z}\in C_{012},
\end{cases}\\
\tau_5(\vec{z})&=
\begin{cases}
\beta \vec{z}-\vec{q}_i, &\mbox{if }\vec{z}\in  E_i,i\in\{0,1,2\}\\
\beta \vec{z}-\vec{q}_j, &\mbox{if }\vec{z}\in  C_{ij},ij\in\{01,02,12\}\\
\beta \vec{z}-\vec{q}_1, &\mbox{if }\vec{z}\in C_{012},
\end{cases}\\
\tau_6(\vec{z})&=
\begin{cases}
\beta \vec{z}-\vec{q}_i, &\mbox{if }\vec{z}\in  E_i,i\in\{0,1,2\}\\
\beta \vec{z}-\vec{q}_j, &\mbox{if }\vec{z}\in  C_{ij},ij\in\{01,02,12\}\\
\beta \vec{z}-\vec{q}_2, &\mbox{if }\vec{z}\in C_{012},
\end{cases}
\end{align*}
The probabilities are defined as follows.
\begin{alignat}{2}
&p_1(\vec{z})= p\cdot s,
& &p_4(\vec{z})=(1-p)\cdot s, \notag\\
&p_2(\vec{z})=p\cdot t,
& & p_5(\vec{z})=(1-p)\cdot t, \label{pk}\\
&p_3(\vec{z})= p\cdot(1-s-t),
 &\quad&p_6(\vec{z})=(1-p)\cdot (1-s-t).\notag
\end{alignat}

We have the following lemma.
\begin{lemma}  \label{prob}
For any $\vec{z}\in S_\beta$,
$\sum_{k_1k_2\cdots k_n\in\{1,\ldots,6\}^n} p_{k_1k_2\ldots k_n}(\vec{z})=1,n\in\mathbb{Z}$.
\end{lemma}
\begin{proof}
We prove this lemma by induction.
For $n=1$, $p_1(\vec{z})+\dots+p_6(\vec{z})=1$.
Assume it is true for $n=m$, i.e. for any $\vec{z}\in S_\beta$,
$$\sum_{k_1k_2\cdots k_m\in\{1,\ldots,6\}^n} p_{k_1k_2\ldots k_m}(\vec{z})=1.$$
 For  $n=m+1$,
\begin{align*}
&\sum_{k_1k_2\cdots k_{m+1}\in\{1,\ldots,6\}^{m+1}} p_{k_1k_2\ldots k_{m+1}}(\vec{z})\\
=&\sum_{k_1k_2\cdots k_{m+1}\in\{1,\ldots,6\}^{m+1}}p_{k_{m+1}}(\tau_{k_m}\circ\cdots\circ\tau_{k_1}(\vec{z}))\cdot p_{k_m}(\tau_{k_{m-1}}\circ\cdots\circ\tau_{k_1}(\vec{z}))\cdots p_{k_1}(\vec{z})\\
=&\sum_{1k_2\cdots k_{m+1}\in\{1,\ldots,6\}^{m+1}} p_{k_{m+1}}(\tau_{k_m}\circ \cdots \circ \tau_{k_2}(\tau_1(\vec{z})))\cdots p_{k_2}(\tau_1(\vec{z}))\cdot p_1(\vec{z})\\
&+\sum_{2k_2\cdots k_{m+1}\in\{1,\ldots,6\}^{m+1}} p_{k_{m+1}}(\tau_{k_m}\circ \cdots \circ \tau_{k_2}(\tau_2(\vec{z})))\cdots p_{k_2}(\tau_2(\vec{z}))\cdot p_2(\vec{z})\\
&+\cdots +\sum_{6k_2\cdots k_{m+1}\in\{1,\ldots,6\}^{m+1}} p_{k_{m+1}}(\tau_{k_m}\circ \cdots \circ \tau_{k_2}(\tau_6(\vec{z})))\cdots p_{k_2}(\tau_6(\vec{z}))\cdot p_6(\vec{z})\\
=&p_1(\vec{z})\sum_{k_2\cdots k_{m+1}\in\{1,\ldots,6\}^{m}} p_{k_2\ldots k_{m+1}}(\tau_1(\vec{z})) + p_2(\vec{z})\sum_{k_2\cdots k_{m+1}\in\{1,\ldots,6\}^{m}} p_{k_2\ldots k_{m+1}}(\tau_2(\vec{z})) \\
&+\cdots+p_6(\vec{z})\sum_{k_2\cdots k_{m+1}\in\{1,\ldots,6\}^{m}} p_{k_2\ldots k_{m+1}}(\tau_6(\vec{z}))  \\
=&p_1(\vec{z})+\dots+p_6(\vec{z})\\
=&1.
\end{align*}
\end{proof}
Now we can prove the existence of an acim for $R$.
\begin{theorem}\label{Racim}
Let $R= \{\tau_1,\ldots,\tau_6;p_1(\vec{z}),\ldots,p_6(\vec{z})\}$, then $R$ admits an acim.
\end{theorem}
\begin{proof}
Denote the  partition (\ref{partition}) by $\mathcal{P}$ with
$$S_1=E_0,S_2=E_1,S_3=E_2, S_4=C_{01},S_5=C_{12},S_6=C_{02},S_7=C_{012}.$$
Consider the iteration of the random map, $R^n$,
the corresponding partition is $\vee_{i=0}^{n-1} R^{-i}\mathcal{P}$, where
$$ R^{-i}\mathcal{P}=\vee_{k_1k_2\cdots k_i\in\{1,\ldots,6\}^i} \tau_{k_1k_2\ldots k_i}^{-1}\mathcal{P}.$$
 For a set $P_i\in \vee_{i=0}^{n-1} R^{-i}\mathcal{P}$ and a sequence $k_1\ldots k_n\in\{1,\ldots,6\}^n$, let $\tau_{k_1\ldots k_n,i}=\tau_{k_1\ldots k_n}|_{P_i}$ and $M_{k_1\ldots k_n,i}=\sup_{\vec{z}\in P_i}(Dp_{k_1\ldots k_n}(\vec{z})-\frac{DJ_{k_1\ldots k_n,i}}{J_{k_1\ldots k_n,i}}p_{k_1\ldots k_n,i}(\vec{z}))$, where $J_{k_1\ldots k_n,i}$ is the Jacobian of $\tau_{k_1\ldots k_n,i}$.
Let
$$M_n=\sum_{k_1\ldots k_n\in\{0,1,2\}^n}\max_{P_i\in\vee_{i=0}^{n-1} R^{-i}\mathcal{P}}M_{k_1\ldots k_n,i} \ \ \text{   and   }\ \ \delta_n=\min_{P_i\in\vee_{i=0}^{n-1} R^{-i}\mathcal{P}}\delta(P_i).$$
For any set $P_i\in \vee_{i=0}^{n-1} R^{-i}\mathcal{P}$, the
derivative matrix of $\tau_{k_1k_2\ldots k_n}^{-1}$ is equal to
$$\begin{bmatrix}
              \frac{1}{\beta^n} & 0 \\
              0 & \frac{1}{\beta^n} \\
    \end{bmatrix}. $$
Using Lemma \ref{prob} we have
\begin{align*}
\max_{P_i\in\vee_{i=0}^{n-1} R^{-i}\mathcal{P}}\sum_{k_1k_2\cdots k_n\in\{1,\ldots,6\}^n} p_{k_1k_2\ldots k_n}(\vec{z})\Vert D(\tau_{k_1k_2\ldots k_n}|_{P_i})^{-1}\Vert=\frac{\sqrt{2}}{\beta^n}<\frac{2\sqrt{2}}{\beta^n}:=c_n.
\end{align*}
For the partition $\vee_{i=0}^{n-1} R^{-i}\mathcal{P}$, we have $a_n=\sqrt{2}/2$.
Let
$$r_n=c_n(1+\frac{1}{a_n})=\frac{2\sqrt{2}+4}{\beta^n},\ \ \  R_n=M_n+\frac{c_n}{a_n\delta_n} .$$
We can find $l>\log(2\sqrt{2}+4)/\log\beta$ such that $r_l<1$.
Fix this $l$ and let $C_1=\max\{r_1,r_2,\ldots,r_{l-1}\}, C_2=\max\{R_1,R_2,\ldots,R_{l-1}\}$.
For any integer $n$, we have $n=jl+i$, where $0\leq i\leq l-1.$
Notice that $P_{R^n}=(P_{R^l})^jP_{R^i}$. Apply Theorem \ref{VF} on $R^l$, then we get
\begin{align*}
V(P_{R^n}f)&=VP_{R^l}^j(P_{R^i}f)\\
&\leq r_l\cdot VP_{R^l}^{j-1}(P_{R^i}f)+R_l\Vert f\Vert_1\\
&\leq r_l\cdot( r_l\cdot VP_{R^l}^{j-2}(P_{R^i}f)+R_l\Vert f\Vert_1)+R_l\Vert f\Vert_1\\
&\cdots\\
&\leq r_l^j V(P_{R^i}f)+(r_l^{j-1}+r_l^{j-2}+\cdots+r_l+1)R_l\Vert f\Vert_1\\
&\leq r_l^j (C_1V(f)+C_2\Vert f\Vert_1)+(r_l^{j-1}+r_l^{j-2}+\cdots+r_l+1)R_l\Vert f\Vert_1\\
&= C_1r_l^j V(f)+(C_2r_l^j+r_l^{j-1}+r_l^{j-2}+\cdots+r_l+1)R_l\Vert f\Vert_1\\
&\leq C_1r_l^j V(f)+(C_2+\frac{1}{1-r_l})R_l\Vert f\Vert_1.
\end{align*}
By definition of the norm $\Vert\cdot\Vert_{BV}$,
\begin{align*}
\Vert P_{R^n}f \Vert_{BV}&=\Vert P_{R^n}f \Vert_1+V(P_{R^n}f)\\
&\leq \Vert f\Vert_1+C_1r_l^j V(f)+(C_2+\frac{1}{1-r_l})R_l\Vert f\Vert_1.
\end{align*}
 Then the result follows by the technique in \cite[Theorem 1]{GB}. We write some details for completeness.
From the above inequality it follows that the set $\{P_R^n\mathbf{1}\}_{n\geq l}$ is
uniformly bounded, where $\mathbf{1}$ is the constant function equal to $1$ on $S_\beta$. Hence $P_R$ has a nontrivial fixed point $\mathbf{1}^*$ which is the density of an acim by the Kakutani-Yoshida Theorem.
\end{proof}
~\\
\textbf{Step 2: for the skew product transformation $R'$ on $S_\beta\times [0,1)$.}

\medskip
Let $(I,\mathcal{B}(I),\lambda_1)$ be the unit interval $I=[0,1)$, with $\mathcal{B}(I)$ the Borel $\sigma$-algebra on $I$ and $\lambda_1$ being Lebesgue measure on $(I,\mathcal{B}(I))$.
Let $Y=S_\beta\times I$ and the set $J_k$ be given by $J_k = \{(\vec{z},w): \sum_{i<k}p_i(\vec{z})\leq w < \sum_{i\leq k}p_i(\vec{z})\}$. Define maps $\varphi_k: J_k\rightarrow I $ by
$$\varphi_k(\vec{z},w)=\frac{1}{p_k(\vec{z})}w-\frac{\sum_{r=1}^{k-1}p_r(\vec{z})}{p_k(\vec{z})}.$$
Define the skew product transformation $R':S_\beta\times I\rightarrow S_\beta\times I $ by
$$R'(\vec{z},w)=(\tau_k(x),\varphi_k(\vec{z},w))$$
for $(\vec{z},w)\in J_k$.

Since $p_k(\vec{z})$ is defined as in (\ref{pk}), then we have
\begin{alignat}{2}
&\varphi_1(\vec{z},w)=\frac{w}{ps},
& &\varphi_4(\vec{z},w)=\frac{w-p}{(1-p)s}, \notag\\
&\varphi_2(\vec{z},w)=\frac{w-ps}{pt},
& & \varphi_5(\vec{z},w)=\frac{w-p-(1-p)s}{(1-p)t}, \notag\\
&\varphi_3(\vec{z},w)=\frac{w-ps-pt}{p(1-s-t)},
 &\quad& \varphi_6(\vec{z},w)=\frac{w-p-(1-p)s-(1-p)t}{(1-p)(1-s-t)}.\notag
\end{alignat}
We denote $p_k(\vec{z})$ and $\varphi_k(\vec{z},w)$ by $p_k$ and $\varphi_k(w)$, respectively, since each $p_k(\vec{z})$ is a constant.
Therefore,
\begin{equation*}
R'(\vec{z},w)=
\begin{cases}
(\tau_1(\vec{z}),\varphi_1(w)), &\mbox{if } w\in[0,ps),\\
(\tau_2(\vec{z}),\varphi_2(w)), &\mbox{if } w\in[ps,ps+pt),\\
(\tau_3(\vec{z}),\varphi_3(w)), &\mbox{if } w\in[ps+pt,p),\\
(\tau_4(\vec{z}),\varphi_4(w)), &\mbox{if } w\in[p,p+(1-p)s),\\
(\tau_5(\vec{z}),\varphi_5(w)), &\mbox{if } w\in[p+(1-p)s,p+(1-p)s+(1-p)t),\\
(\tau_6(\vec{z}),\varphi_6(w)), &\mbox{if } w\in[p+(1-p)s+(1-p)t,1).\\
\end{cases}
\end{equation*}

Denote by $\mu_\beta$ an acim for the position dependent random transformation $R=\{\tau_1,\ldots,\tau_6;p_1,\ldots,p_6\}$, which means $\mu_\beta$ is $R$-invariant and absolutely continuous with respect to Lebesgue measure $\lambda_2$ in $\mathbb{R}^2$.
We start by recalling Lemma 3.2 in \cite{BBQ}.

\begin{lemma}\label{RtoR'}
$\mu_\beta$ is invariant for the random map $R$ if and only if $\mu_\beta\otimes\lambda_1$ is invariant for the skew product $R'$.
\end{lemma}
~\\
\textbf{Step 3: for the skew product transformation $R_\beta$ on $\Omega\times\Upsilon\times S_\beta$ .}

\medskip
Define the $skew\ product\ transformation\ R_\beta$ on $\Omega\times\Upsilon\times S_\beta$ as follows:
\begin{equation*}
R_\beta(\omega,\upsilon,\vec{z})=
\begin{cases}
(\sigma\omega,\sigma'\upsilon, \beta \vec{z}-\vec{q}_i), &\mbox{if }\vec{z}\in  E_i,i\in\{0,1,2\}\\
(\sigma\omega,\sigma'\upsilon, \beta \vec{z}-\vec{q}_i), &\mbox{if }\vec{z}\in  C_{ij},ij\in\{01,02,12\} \mbox{ and }\omega_1=0\\
(\sigma\omega,\sigma'\upsilon, \beta \vec{z}-\vec{q}_j), &\mbox{if }\vec{z}\in  C_{ij},ij\in\{01,02,12\} \mbox{ and }\omega_1=1\\
(\sigma\omega,\sigma'\upsilon, \beta \vec{z}-\vec{q}_i), &\mbox{if }\vec{z}\in C_{012},\upsilon_1=i,i\in\{0,1,2\}.
\end{cases}
\end{equation*}
\begin{lemma}\label{R'toRbeta}
 $(S_\beta \times I, \mathcal{S\times B}(I),\mu_\beta\otimes \lambda_1,R')$ and $(\Omega\times\Upsilon\times S_\beta,\mathcal{A\times B\times S},m_1\otimes m_2\otimes \mu_\beta,R_\beta)$ are isomorphic.
\end{lemma}
\begin{proof}
Let $\pi_2: S_\beta\times I\rightarrow I$ be the canonical projection onto the second coordinate.
Consider the map $\varphi=\pi_2\circ R'$ on $(I,\mathcal{B}(I),\lambda_1)$. One can see that
$\varphi(w)=\varphi_k(w)$ for $w\in I_k$,
where $I_1=[0,p_1)$ and $I_k=\Big[\sum_{i=1}^{k-1}p_i,\sum_{i=1}^{k}p_i\Big)$ for $2\leq k \leq 6$.
Define
$$l(w)=\frac{1}{p_k}\quad \text{and} \quad h(w)=\frac{\sum_{i=1}^{k-1}p_i}{p_k} $$
for $w\in I_k$. It follows that $\varphi(w)=l(w)\cdot w-h(w)$.
Let
$$l_n=l_n(w):=l(\varphi^{n-1}(w))\quad \text{and} \quad h_n=h_n(w):=h(\varphi^{n-1}(w)).$$
For $w\in[0,1)$, we can write the generalized\ L{\"u}roth\ series\ (GLS) of $w$, which is
$$w=\frac{h_1}{l_1}+\frac{h_2}{l_1l_2}+\cdots+\frac{h_n}{l_1\cdots l_n}+\cdots.$$
Consider the system $\{\{0,1,2,3,4,5\}^{\mathbb{N}}, \mathcal{C},m,\sigma''\}$,
where $\mathcal{C}$ is the product $\sigma$-algebra,
$\sigma''$ is the left shift
and $m$ is the product measure with weights $\{p_1,\ldots,p_6\}$ as (\ref{pk}).
Let $\phi_1:I \rightarrow \{0,1,2,3,4,5\}^{\mathbb{N}} $ be given by
$$\phi_1:w=\sum_{n=1}^\infty\frac{h_i}{l_1l_2\cdots l_i}\mapsto (\gamma_1,\gamma_2,\ldots,),$$
where $\gamma_n=\gamma_n(w),n\geq 1$ is defines as follows:
$$\gamma_n:=\gamma_n(w)=k-1 \iff \varphi^{n-1}(w)\in I_k,$$
for $k\in\{1,2,3,4,5,6\}.$
It is known that $\varphi$ preserves the Lebesgue measure $\lambda_1$ and $\phi_1$ is an isomorphism between the two dynamical systems $(I,\mathcal{B}(I),\lambda_1,\varphi\}$ and $\{\{0,1,2,3,4,5\}^{\mathbb{N}}, \mathcal{C},m,\sigma''\}$. See \cite{BBD} for more details.


Next we give a map $\phi_2$ from $\{\{0,1,2,3,4,5\}^{\mathbb{N}}, \mathcal{C},m,\sigma''\}$ to $\{\Omega\times\Upsilon,\mathcal{A\times B},m_1\otimes m_2,\sigma\times\sigma'\}$.
Let $h_1:\{0,1,2,3,4,5\}\rightarrow\{0,1\}$ and $h_2:\{0,1,2,3,4,5\}\rightarrow\{0,1,2\}$ be given by
\begin{equation*}
h_1(x)=
\begin{cases}
0, &\mbox{if } x=0,1,2,\\
1, &\mbox{if } x=3,4,5,\\
\end{cases},
\quad
h_2(x)=
\begin{cases}
0, &\mbox{if } x=0,3,\\
1, &\mbox{if } x=1,4,\\
2, &\mbox{if } x=2,5.\\
\end{cases}
\end{equation*}
Define $\phi_2:\{0,1,2,3,4,5\}^{\mathbb{N}}\rightarrow\Omega\times\Upsilon$ by
$\phi_2(\gamma)=(\omega,\upsilon),$
where
\begin{align*}
&\omega=(h_1(\gamma_1),h_1(\gamma_2),h_1(\gamma_3),\ldots):=\tilde{h}_1(\gamma),\\
&\upsilon=(h_2(\gamma_1),h_2(\gamma_2),h_2(\gamma_3),\ldots):=\tilde{h}_2(\gamma).
\end{align*}
One can see that $\phi_2$ maps a cylinder of rank $n$ in $\{0,1,2,3,4,5\}^{\mathbb{N}}$ to the product of two cylinders of the same rank $n$ in $\Omega\times\Upsilon$. It follows that $\phi_2$ is a bimeasurable bijection. From the definition of the product measure, we can get the measure preservingness on cylinders. Finally, it is easy to see that $\phi_2\circ \sigma''=(\sigma\times\sigma')\circ \phi_2$.
Therefore, $\phi_2$ is an isomorphism.

Now let $\phi:S_\beta \times I\rightarrow \Omega\times\Upsilon\times S_\beta $ be given by
\begin{equation*}
\phi(\vec z,w)=(\tilde{h}_1(\phi_1(w)),\tilde{h}_2(\phi_1(w)),\vec z).
\end{equation*}
In fact, $\phi=\iota\circ( I_{S_\beta}\times(\phi_2\circ\phi_1))$,
where $I_{S_\beta}$ is the identity map on $S_\beta$
and $\iota(\vec z,\omega,\upsilon)=(\omega,\upsilon,\vec z)$ is a transformation that only changes the order of coordinates.
Since $\phi_2\circ \phi_1$ preserves the dynamics of $\pi_2\circ R$ and $\sigma\times\sigma'$, i.e.,
$$(\phi_2\circ \phi_1) \circ (\pi_2\circ R)=(\sigma\times\sigma')\circ(\phi_2\circ \phi_1),$$
we have that $\phi\circ R'= R_\beta\circ \phi$.
Therefore, the result follows.
\end{proof}
~\\
\textbf{Step 4: for the random transformation $K_\beta$ on $\Omega\times\Upsilon\times S_\beta$. }

\medskip
Define a skew product transformation $R_\beta$ as follows:
\begin{equation*}
R_\beta(\omega,\upsilon,\vec{z})=
\begin{cases}
(\sigma\omega,\sigma'\upsilon,\tau_1(\vec{z})), &\mbox{if } \omega_1=0,\upsilon_1=0,\\
(\sigma\omega,\sigma'\upsilon,\tau_2(\vec{z})), &\mbox{if } \omega_1=0,\upsilon_1=1,\\
(\sigma\omega,\sigma'\upsilon,\tau_3(\vec{z})), &\mbox{if } \omega_1=0,\upsilon_1=2,\\
(\sigma\omega,\sigma'\upsilon,\tau_4(\vec{z})), &\mbox{if } \omega_1=1,\upsilon_1=0,\\
(\sigma\omega,\sigma'\upsilon,\tau_5(\vec{z})), &\mbox{if } \omega_1=1,\upsilon_1=1,\\
(\sigma\omega,\sigma'\upsilon,\tau_6(\vec{z})), &\mbox{if } \omega_1=1,\upsilon_1=2,\\
\end{cases}
\end{equation*}
Let $\mu$ be an arbitrary probability measure on $S_\beta$.
We will show that any product measure of the form $m_1\otimes m_2\otimes \mu$ is $K_\beta$-invariant if
and only if it is $R_\beta$-invariant.

\begin{lemma}\label{RbetatoKbeta}
$m_1\otimes m_2 \otimes \mu \circ K_\beta^{-1}=m_1\otimes m_2 \otimes \mu \circ R_\beta^{-1}=m_1\otimes m_2\otimes \nu$, where
\begin{align*}
\nu=&p s\cdot \mu\circ \tau_1^{-1}+pt\cdot \mu\circ \tau_2^{-1}+p(1-s-t)\cdot \mu\circ \tau_3^{-1}\\
&+(1-p)\cdot s\cdot \mu\circ \tau_4^{-1}+(1-p)\cdot t\cdot \mu\circ \tau_5^{-1}+(1-p)\cdot (1-s-t)\cdot \mu\circ \tau_6^{-1}.
\end{align*}
\end{lemma}
\begin{proof}
Denote by $C_1$ and $C_2$ arbitrary cylinders in $\Omega$ and $\Upsilon$, respectively. Let $S$ be a closed set in $S_\beta$. It suffices to verify that the measures coincide on sets of the form $C_1\times C_2 \times S$, because the collection of these sets forms a generating $\pi$-system. Let $[i,C_1]=\{\omega_1=i\}\cap\sigma^{-1}(C_1)$ for $i=0,1$ and $[i,C_2]=\{\upsilon_1=i\}\cap(\sigma')^{-1}(C_2)$ for $i=0,1,2$. Notice that
\begin{align*}
&\tau_1(S)\cap E=\tau_2(S)\cap E=\tau_3(S)\cap E=\tau_4(S)\cap E=\tau_5(S)\cap E=\tau_6(S)\cap E,\\
&\tau_1(S)\cap C=\tau_2(S)\cap C=\tau_3(S)\cap C,\\
&\tau_4(S)\cap C=\tau_5(S)\cap C=\tau_6(S)\cap C,\\
&\tau_1(S)\cap C_{012}=\tau_4(S)\cap C_{012},\\
&\tau_2(S)\cap C_{012}=\tau_5(S)\cap C_{012},\\
&\tau_3(S)\cap C_{012}=\tau_6(S)\cap C_{012}.
\end{align*}
We can get
\begin{align*}
K_\beta^{-1}(C_1\times C_2 \times S)
=&C_1\times C_2 \times (\tau_1(S)\cap E)\\
 &\cup [0,C_1]\times C_2 \times (\tau_1(S)\cap C)\cup [1,C_1]\times C_2 \times (\tau_4(S)\cap C)\\
&\cup C_1\times [0,C_2] \times (\tau_1(S)\cap C_{012})\cup C_1\times [1,C_2] \times (\tau_2(S)\cap C_{012})\\
&\cup C_1\times [2,C_2] \times (\tau_3(S)\cap C_{012})
\end{align*}
Hence,
\begin{align*}
&m_1\otimes m_2 \otimes \mu \circ K_\beta^{-1}(C_1\times C_2 \times S)\\
=&m_1(C_1)\cdot m_2( C_2) \cdot \mu(\tau_1(S)\cap E)\\
&+p\cdot m_1(C_1)\cdot m_2( C_2) \cdot \mu(\tau_1(S)\cap C)+(1-p)\cdot m_1(C_1)\cdot m_2( C_2) \cdot \mu(\tau_4(S)\cap C)\\
&+s\cdot m_1(C_1)\cdot m_2( C_2) \cdot \mu(\tau_1(S)\cap  C_{012})+t\cdot m_1(C_1)\cdot m_2( C_2) \cdot \mu(\tau_2(S)\cap  C_{012})\\
&+(1-s-t)\cdot m_1(C_1)\cdot m_2( C_2) \cdot \mu(\tau_3(S)\cap  C_{012}\\
=&ps\cdot m_1(C_1)\cdot m_2( C_2) \cdot \mu(\tau_1(S))+pt\cdot m_1(C_1)\cdot m_2( C_2) \cdot \mu(\tau_2(S))\\
&+p(1-s-t)\cdot m_1(C_1)\cdot m_2(C_2) \cdot \mu(\tau_3(S))+(1-p)s\cdot m_1(C_1)\cdot m_2( C_2) \cdot\mu(\tau_4(S))\\
&+(1-p)t\cdot m_1(C_1)\cdot m_2(C_2) \cdot \mu(\tau_5(S))+(1-p)(1-s-t)\cdot m_1(C_1)\cdot m_2(C_2) \cdot\mu(\tau_6(S))\\
=&m_1\otimes m_2\otimes \nu(C_1\times C_2\times S).
\end{align*}
On the other hand,
\begin{align*}
&R_\beta^{-1}(C_1\times C_2 \times S)\\
=&[0,C_1]\times [0,C_2] \times \tau_1(S)\cup [0,C_1]\times [1,C_2] \times \tau_2(S)\\
&\cup [0,C_1]\times [2,C_2] \times \tau_3(S)\cup [1,C_1]\times [0,C_2] \times \tau_4(S)\\
&\cup [1,C_1]\times [1,C_2] \times \tau_5(S)\cup [1,C_1]\times [2,C_2] \times \tau_6(S).
\end{align*}
Therefore, we complete the proof.
\end{proof}

Now we give the main result in this section.
\begin{theorem}
Let $\beta \in (1,3/2)$. Then $K_\beta$ has an invariant measure of the form $m_1 \otimes m_2 \otimes \mu_\beta$,
where $\mu_\beta$ is absolutely continuous with respect to $\lambda_2$.
\end{theorem}
\begin{proof}
By Theorem \ref{Racim}, Lemma \ref{RtoR'}, Lemma \ref{R'toRbeta}, and Lemma \ref{RbetatoKbeta}, we complete the proof.
\end{proof}

\begin{remark}
When $\beta=3/2$, $C_{012}=\{(\frac{2}{3},\frac{2}{3})\}$ is a point. We modify the definition of $R,R',R_\beta$, and give relevant conclusions.\\
(i) Let $R= \{\tau_1,\tau_2;p_1(\vec{z}),p_2(\vec{z})\}$ be a position dependent random transformation on $S_\beta$, where
\begin{align*}
\tau_1(\vec{z})&=
\begin{cases}
\beta \vec{z}-\vec{q}_i, &\mbox{if }\vec{z}\in  E_i,i\in\{0,1,2\},\\
\beta \vec{z}-\vec{q}_i, &\mbox{if }\vec{z}\in  C_{ij},ij\in\{01,02,12\},\\
\beta \vec{z}, &\mbox{if }\vec{z}=(\frac{2}{3},\frac{2}{3}),
\end{cases}\\
\tau_2(\vec{z})&=
\begin{cases}
\beta \vec{z}-\vec{q}_i, &\mbox{if }\vec{z}\in  E_i,i\in\{0,1,2\},\\
\beta \vec{z}-\vec{q}_j, &\mbox{if }\vec{z}\in  C_{ij},ij\in\{01,02,12\},\\
\beta \vec{z}, &\mbox{if }\vec{z}=(\frac{2}{3},\frac{2}{3}),
\end{cases}
\end{align*}
and $p_1(\vec{z})=p_2(\vec{z})=1/2$ for $\vec z \in S_\beta$. Similar to Theorem \ref{Racim}, it is not difficult to prove that $R$ has an acim $\mu_\beta$.\\
(ii) By Lemma 3.2 in \cite{BBQ}, $\mu_\beta\otimes\lambda_1$ is invariant for the skew product $R'$, where
\begin{equation*}
R'(\vec{z},w)=
\begin{cases}
(\tau_1(\vec{z}),\frac{w}{p}), &\mbox{if } w\in[0,p),\\
(\tau_2(\vec{z}),\frac{w-p}{1-p}), &\mbox{if } w\in[p,1).\\
\end{cases}
\end{equation*}
\\
(iii) Define the $skew\ product\ transformation\ R_\beta$ on $\Omega\times S_\beta$ as follows:
\begin{equation*}
R_\beta(\omega,\vec{z})=
\begin{cases}
(\sigma\omega, \beta \vec{z}-\vec{q}_i), &\mbox{if }\vec{z}\in  E_i,i\in\{0,1,2\},\\
(\sigma\omega, \beta \vec{z}-\vec{q}_i), &\mbox{if }\vec{z}\in  C_{ij},ij\in\{01,02,12\},\omega_1=0,\\
(\sigma\omega, \beta \vec{z}-\vec{q}_j), &\mbox{if }\vec{z}\in  C_{ij},ij\in\{01,02,12\},\omega_1=1,\\
(\sigma\omega, \beta \vec{z}), &\mbox{if }\vec{z}=(\frac{2}{3},\frac{2}{3}).
\end{cases}
\end{equation*}
Then we have that the dynamical systems
 $(S_\beta \times I, \mathcal{S\times B}(I),\mu_\beta\otimes \lambda_1,R')$ and $(\Omega\times S_\beta,\mathcal{A\times S},m_1\otimes \mu_\beta,R_\beta)$ are isomorphic.
The proof is similar and easier than that of Lemma \ref{R'toRbeta}.\\
(iv) Let $\mu$ be an arbitrary probability measure on $S_\beta$
and let $\tilde K_\beta: \Omega\times S_\beta\rightarrow \Omega \times S_\beta$ be given by
\begin{equation*}
\tilde K_\beta(\omega,\vec{z})=
\begin{cases}
(\omega, \beta \vec{z}-\vec{q}_i), &\mbox{if }\vec{z}\in  E_i,i\in\{0,1,2\}\\
(\sigma\omega, \beta \vec{z}-\vec{q}_i), &\mbox{if }\vec{z}\in  C_{ij},ij\in\{01,02,12\},\omega_1=0\\
(\sigma\omega, \beta \vec{z}-\vec{q}_j), &\mbox{if }\vec{z}\in  C_{ij},ij\in\{01,02,12\},\omega_1=1\\
(\omega, \beta \vec{z}), &\mbox{if }\vec{z}=(\frac{2}{3},\frac{2}{3}).
\end{cases}
\end{equation*}
It is easy to check that
$$m_1\otimes \mu \circ \tilde K_\beta^{-1}=m_1 \otimes \mu \circ R_\beta^{-1}=m_1\otimes \nu,$$
where
$$\nu=p \cdot \mu\circ \tau_1^{-1}+(1-p)\cdot \mu\circ \tau_2^{-1}$$
by using the same method of calculation in Lemma \ref{RbetatoKbeta}. Therefore, it follows from (i-iv) that $\tilde K_\beta$ has an invariant measure of the form $ m_1\otimes\mu_\beta$.
\end{remark}

Next we give another method to prove that $R_\beta$ has an acim by using Corollary \ref{tacim}.
Denote $I_0=[0,p),I_1= [p,1],J_0=[0,s),J_1=[s,s+t),J_2=[s+t,1].$
Let
\begin{align*}
\tilde {R}(\vec{z},w,v)&=
\begin{cases}
(\beta \vec{z}-\vec{q}_i,\frac{w}{p},\frac{v}{s}) &\mbox{if }\vec{z}\in  E_i\times [0,p)\times [0,s),i\in\{0,1,2\},\\
(\beta \vec{z}-\vec{q}_i,\frac{w-p}{1-p},\frac{v}{s}) &\mbox{if }\vec{z}\in  E_i\times [p,1]\times [0,s),i\in\{0,1,2\},\\
(\beta \vec{z}-\vec{q}_i,\frac{w}{p},\frac{v-s}{t}) &\mbox{if }\vec{z}\in  E_i\times [0,p)\times [s,s+t),i\in\{0,1,2\},\\
(\beta \vec{z}-\vec{q}_i,\frac{w-p}{1-p},\frac{v-s}{t}) &\mbox{if }\vec{z}\in  E_i\times [p,1]\times [s,s+t),i\in\{0,1,2\},\\
(\beta \vec{z}-\vec{q}_i,\frac{w}{p},\frac{v-s-t}{1-s-t}) &\mbox{if }\vec{z}\in  E_i\times [0,p)\times [s+t,1],i\in\{0,1,2\},\\
(\beta \vec{z}-\vec{q}_i,\frac{w-p}{1-p},\frac{v-s-t}{1-s-t}) &\mbox{if }\vec{z}\in  E_i\times [p,1]\times [s+t,1],i\in\{0,1,2\},\\
(\beta \vec{z}-\vec{q}_i,\frac{w}{p},\frac{v}{s}) &\mbox{if }\vec{z}\in  C_{ij}\times [0,p)\times [0,s),ij\in\{01,02,12\},\\
(\beta \vec{z}-\vec{q}_j,\frac{w-p}{1-p},\frac{v}{s}) &\mbox{if }\vec{z}\in  C_{ij}\times [p,1]\times [0,s),ij\in\{01,02,12\},\\
(\beta \vec{z}-\vec{q}_i,\frac{w}{p},\frac{v-s}{t}) &\mbox{if }\vec{z}\in  C_{ij}\times [0,p)\times [s,s+t),ij\in\{01,02,12\},\\
(\beta \vec{z}-\vec{q}_j,\frac{w-p}{1-p},\frac{v-s}{t}) &\mbox{if }\vec{z}\in  C_{ij}\times [p,1]\times [s,s+t),ij\in\{01,02,12\},\\
(\beta \vec{z}-\vec{q}_i,\frac{w}{p},\frac{v-s-t}{1-s-t}) &\mbox{if }\vec{z}\in  C_{ij}\times [0,p)\times [s+t,1],ij\in\{01,02,12\},\\
(\beta \vec{z}-\vec{q}_j,\frac{w-p}{1-p},\frac{v-s-t}{1-s-t}) &\mbox{if }\vec{z}\in  C_{ij}\times [p,1]\times [s+t,1],ij\in\{01,02,12\},\\
(\beta \vec{z},\frac{w}{p},\frac{v}{s}) &\mbox{if }\vec{z}\in C_{012}\times [0,p)\times [0,s),\\
(\beta \vec{z},\frac{w-p}{1-p},\frac{v}{s}) &\mbox{if }\vec{z}\in C_{012}\times [p,1]\times [0,s),\\
(\beta \vec{z},\frac{w}{p},\frac{v-s}{t}) &\mbox{if }\vec{z}\in C_{012}\times [0,p)\times [s,s+t),\\
(\beta \vec{z},\frac{w-p}{1-p},\frac{v-s}{t}) &\mbox{if }\vec{z}\in C_{012}\times [p,1]\times [s,s+t),\\
(\beta \vec{z},\frac{w}{p},\frac{v-s-t}{1-s-t}) &\mbox{if }\vec{z}\in C_{012}\times [0,p)\times [s+t,1],\\
(\beta \vec{z},\frac{w-p}{1-p},\frac{v-s-t}{1-s-t}) &\mbox{if }\vec{z}\in C_{012}\times [p,1]\times [s+t,1].
\end{cases}\\
\end{align*}
We can get the following theorem.
\begin{lemma}\label{R}
 Let $\beta\in(1,3/2)$, then $\tilde R$ admits an acim.
\end{lemma}
\begin{proof}

Use the partition $\mathcal{P}:\{E_i\times I_m\times J_n,C_{ij}\times I_m\times J_n,C_{012}\times I_m\times J_n,\},i=0,1,2;m=0,1;n=0,1,2$.

Consider the iterate $\tilde R^k$ and the corresponding partition $\vee_{i=0}^{k-1}\tilde R^{-i}(\mathcal{P})$. For $P_i\in \vee_{i=0}^{k-1}\tilde R^{-i}(\mathcal{P})$, let $\tilde R^k_i=\tilde R^k|_{P_i}$.
Since the derivative matrix of $(\tilde R_i^{k})^{-1}$ is
$$\begin{bmatrix}
              \frac{1}{\beta^k} & 0& 0& 0 \\
              0 & \frac{1}{\beta^k} & 0& 0 \\
              0 & 0 &p^l(1-p)^{k-l} &  0 \\
              0  & 0& 0& s^i t^j (1-s-t)^{k-i-j} \\
    \end{bmatrix}, $$
where $l,i,j\in\{0,1,\ldots,k\}$ and $0\leq i+j\leq k$, then the Euclidean matrix norm,
$$\Vert D(\tilde R_i^k)^{-1}\Vert=\sqrt{\frac{2}{\beta^{2k}}+p^{2l}(1-p)^{2(k-l)}+s^{2i} t^{2j} (1-s-t)^{2(k-i-j)}}.$$
Since $\beta>1$ and $0<p,s,t,s+t<1 $, then $\lim_{n\rightarrow\infty}\Vert D(\tilde R_i^k)^{-1}\Vert=0$ for any $P_i\in \vee_{i=0}^{k-1}\tilde R^{-i}(\mathcal{P})$. It means that we can choose an integer $k$ such that $$\Vert D(\tilde R_i^k)^{-1}\Vert<c$$  for any $0<c<1$.
For the partition $\vee_{i=0}^{k-1}\tilde R^{-i}(\mathcal{P})$,
$a$ is a positive constant which is not larger than 1 since any set in the partition is a polyhedra in $\mathbb{R}^4$.
 Let $c<\frac{a}{a+1}$, then we have
$$c(1 + \frac{1}{a})<1.$$
By Corollary \ref{tacim}, $\tilde R$ admits an acim.
\end{proof}
\begin{theorem}
 $R_\beta$ admits an acim.
\end{theorem}
\begin{proof}Set
\begin{equation*}
l'(x)=\begin{cases}
\frac{1}{p}, & \text{ if } x\in[0,p),\\
\frac{1}{1-p}, & \text{ if } x\in[p,1],\\
\end{cases}
\quad
h'(x)=\begin{cases}
0, & \text{ if } x\in[0,p),\\
\frac{p}{1-p}, & \text{ if } x\in[p,1].\\
\end{cases}
\end{equation*}
Let $\varphi'(x)=l'(x)\cdot x-h'(x)$ and
$$l'_n=l'_n(x):=l'(\varphi'^{n-1}(x)),\quad \quad h'_n=h'_n(x):=h'(\varphi'^{n-1}(x)).$$
Define the map $\phi_1:[0,1]\rightarrow \Omega$ as follows:
$$\phi_1:x=\sum_{n=1}^\infty\frac{h'_i}{l'_1l'_2\cdots l'_i}\mapsto (\omega_1,\omega_2,\ldots,),$$
where $\omega_n=\omega_n(x)=k-1 \iff \varphi^{n-1}(x)\in I_k$ for $I_1=[0,p)$ and $I_2=[p,1]$.\\
Similarly, set
\begin{equation*}
l''(x)=\begin{cases}
\frac{1}{s}, & \text{ if } x\in[0,s),\\
\frac{1}{t}, & \text{ if } x\in[s,s+t),\\
\frac{1}{1-s-t}, & \text{ if } x\in[s+t,1],\\
\end{cases}
\quad
h''(x)=\begin{cases}
0, & \text{ if } x\in[0,s),\\
\frac{s}{t}, & \text{ if } x\in[s,s+t),\\
\frac{s+t}{1-s-t}, & \text{ if } x\in[s+t,1].\\
\end{cases}
\end{equation*}
and
$$l''_n=l''_n(x):=l''(\varphi''^{n-1}(x)),\quad \quad h''_n=h''_n(w):=h''(\varphi''^{n-1}(x)),$$
where $\varphi''(x)=l''(x)\cdot x-h''(x)$.
Let $\phi_2:[0,1]\rightarrow \Upsilon$ be given by
$$\phi_2:x=\sum_{n=1}^\infty\frac{h''_i}{l''_1l''_2\cdots l''_i}\mapsto (\upsilon_1,\upsilon_2,\ldots,),$$
where $\upsilon_n=\upsilon_n(x)=k-1 \iff \varphi^{n-1}(x)\in J_k$
for $J_1=[0,s),J_2=[s,s+t)$ and $J_3=[s+t,1]$.\\
Consider the map $\phi=\iota\circ (I_{S_\beta}\times \phi_1\times \phi_2)$
from $(S_\beta \times [0,1] \times [0,1],\mathcal{S\times B}[0,1]\times \mathcal{B}[0,1],\nu,\tilde R)$ to
$(\Omega \times \Upsilon \times S_\beta,\mathcal{A\times B\times S},\mu,R_\beta)$,
where $\iota(\vec z,\omega,\upsilon)=(\omega,\upsilon,\vec z)$ and $I_{S_\beta}$ is the identity map on $S_\beta$.
It is easy to check that $\phi$ is an isomorphism by the definition, and the result follows from Lemma \ref{R}.
\end{proof}

\section{random transformation $K_\beta$ for $\frac{3}{2}<\beta\leq \beta^*$}

\subsection{A special class of $S_\beta$}For $3/2<\beta<2$, we should notice that there are holes in the attractor. Broomhead et al.\cite{BMS} discussed a special structure of those holes: they are all centered on three radial lines originating from the center of the attractor and extending to the three vertices. Since the Sierpinski triangle they discussed can be converted to $S_\beta$ through affine transformation, we can obtain the corresponding range of $\beta$.
Their set-up is as follows. Let $S_\lambda$ be the attractor of an IFS
$$\{ g_i(\vec{z})=\lambda\vec{z}+(1-\lambda)\vec{p}_i\},\  i=0,1,2,$$
where $\vec{p}_i=\frac{2}{3}(\cos\frac{2\pi i}{3},\sin\frac{2\pi i}{3})$, which are the vertices of an equilateral triangle $\Delta'$. Denote the central hole by $H'$, i.e. $H'=\Delta'\setminus \cup_{i=0}^2 g_i(\Delta')$.
Recall the following result in \cite[Proposition 3.7]{BMS}.
\begin{proposition}\label{SS}
Let $\lambda_*\approx 0.6478$ be the appropriate root of
$$x^3-x^2+x=\frac{1}{2}.$$
Then $S_\lambda$ has a non-empty interior if $\lambda\in[\lambda_*,\frac{2}{3})$, and moreover, each hole has the form $g_i^n(H')$, see Figure \ref{radialS1}.
\end{proposition}
\begin{figure}[h]
\centering
\includegraphics[width=10.5cm,height=8cm]{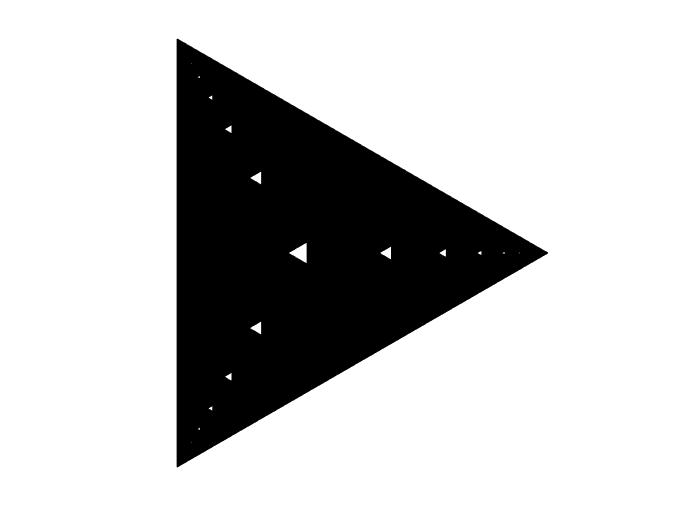}
\caption{$S_\lambda$ for $\lambda_* \leq \lambda < \frac{2}{3}$}
\label{radialS1}
\end{figure}
Notice that the IFS given by (\ref{IFS}) can also be writen as
$$f_{\vec{q}_i}(\vec{z})=\frac{1}{\beta}\vec{z}+(1-\frac{1}{\beta})\frac{\vec{q}_i}{\beta-1},i=0,1,2.$$
Since there is an invertible affine transformation $l$ from $S_\lambda$ to $S_\beta$,
then the two attractors are affinely equivalent when $\lambda=1/\beta$. Therefore we have the following proposition.
\begin{proposition}\label{structure}
Let $\beta^*\approx 1.5437$ be the root of
$$x^3-2x^2+2x=2.$$
Then $S_\beta$ has a non-empty interior if $\beta\in(\frac{3}{2},\beta^*]$, and each hole has the form $f_{
\vec{q}_i}^n(H)$, see Figure \ref{radial}.
\end{proposition}
\begin{figure}[h]
\centering
\includegraphics[width=10.5cm,height=8cm]{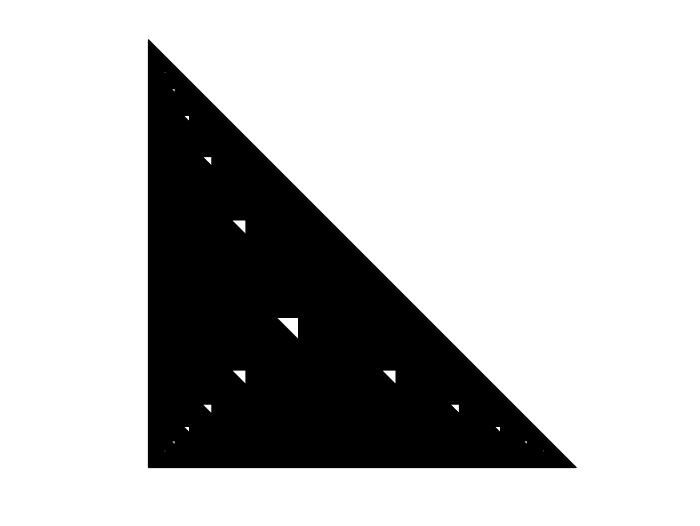}
\caption{$S_\beta$ for $\frac{3}{2}<\beta\leq \beta^*$}
\label{radial}
\end{figure}
\begin{proof}
Let $$l(\vec{z})=\frac{1}{\beta-1}\begin{bmatrix}
              -1/2 & \sqrt{3}/2\\
             -1/2 & -\sqrt{3}/2 \\
    \end{bmatrix}\vec{z}+\frac{1}{\beta-1}\begin{bmatrix}
             1/3\\
             1/3\\
    \end{bmatrix}.$$
Then we have $l(\vec{p}_i)=\frac{1}{\beta-1}\vec{q}_i,i\in\{0,1,2\}$, which implies
\begin{equation}\label{l1}
l(\Delta')=\Delta.
\end{equation}
Since
\begin{align*}
f_{\vec{q}_i}\circ l(\vec{z})
&=\frac{1}{\beta} (\frac{1}{\beta-1}\begin{bmatrix}
              -1/2 & \sqrt{3}/2\\
             -1/2 & -\sqrt{3}/2 \\
    \end{bmatrix}\vec{z}+\frac{1}{\beta-1}\begin{bmatrix}
             1/3\\
             1/3\\
    \end{bmatrix}) +(1-\frac{1}{\beta})\frac{\vec{q}_i}{\beta-1}\\
&=\frac{1}{\beta(\beta-1)}\begin{bmatrix}
              -1/2 & \sqrt{3}/2\\
             -1/2 & -\sqrt{3}/2 \\
    \end{bmatrix}\vec{z}+\frac{1}{\beta(\beta-1)}\begin{bmatrix}
             1/3\\
             1/3\\
    \end{bmatrix} +(1-\frac{1}{\beta})l(\vec{p}_i),
\end{align*}
and
\begin{align*}
l \circ g_i(\vec{z})
&=\frac{1}{\beta-1}\begin{bmatrix}
              -1/2 & \sqrt{3}/2\\
             -1/2 & -\sqrt{3}/2 \\
    \end{bmatrix}(\lambda\vec{z}+(1-\lambda)\vec{p}_i)+\frac{1}{\beta-1}\begin{bmatrix}
             1/3\\
             1/3\\
    \end{bmatrix}\\
&=\frac{\lambda}{\beta-1}\begin{bmatrix}
              -1/2 & \sqrt{3}/2\\
             -1/2 & -\sqrt{3}/2 \\
    \end{bmatrix}\vec{z}+\frac{1-\lambda}{\beta-1}\begin{bmatrix}
              -1/2 & \sqrt{3}/2\\
             -1/2 & -\sqrt{3}/2 \\
    \end{bmatrix}\vec{p}_i+\frac{1}{\beta-1}\begin{bmatrix}
             1/3\\
             1/3\\
    \end{bmatrix}\\
&=\frac{\lambda}{\beta-1}\begin{bmatrix}
              -1/2 & \sqrt{3}/2\\
             -1/2 & -\sqrt{3}/2 \\
    \end{bmatrix}\vec{z}+
(1-\lambda)(l(p_i)-\frac{1}{\beta-1}\begin{bmatrix}
             1/3\\
             1/3\\
    \end{bmatrix})
+\frac{1}{\beta-1}\begin{bmatrix}
             1/3\\
             1/3\\
    \end{bmatrix}\\
&=\frac{\lambda}{\beta-1}\begin{bmatrix}
              -1/2 & \sqrt{3}/2\\
             -1/2 & -\sqrt{3}/2 \\
    \end{bmatrix}\vec{z}+
(1-\lambda)l(p_i)
+\frac{\lambda}{\beta-1}\begin{bmatrix}
             1/3\\
             1/3\\
    \end{bmatrix},
\end{align*}
then
$f_{\vec{q}_i}\circ l=l \circ g_i$ by $\lambda=1/\beta$. From this, we can get
\begin{equation}\label{l2}
l(H')
=l(\Delta')\setminus \cup_{i=0}^2 l(g_i(\Delta'))
=\Delta\setminus \cup_{i=0}^2 f_{\vec{q}_i}(\Delta)
=H.
\end{equation}

If $\beta\in(\frac{3}{2},\beta^*]$, then $\lambda\in[\lambda_*,\frac{2}{3})$. By Proposition \ref{SS}(see \cite[Proposition 3.7]{BMS} for details),
we have that for $n\geq 2$,
\begin{itemize}[leftmargin=*]
\item[](1) $g_ig_j^{n-1}(H')\subset g_j(\Delta')$;
\item[](2) $g_ig_j^{n-1}(H')\cap g_jg_i^{n-1}(H')=\emptyset$.
\end{itemize}
Then it follows from (\ref{l1}) and (\ref{l2}) that
\begin{align*}
(1)f_{\vec{q}_i}f_{\vec{q}_j}^{n-1}(H)&=f_{\vec{q}_i}f_{\vec{q}_j}^{n-1}(l(H')) =f_{\vec{q}_i}f_{\vec{q}_j}^{n-2}(l\circ g_i(H')) =l(g_ig_j^{n-1}(H'))\\
&\subset l\circ g_j(\Delta')=f_{\vec{q}_j}\circ l(\Delta')=f_{\vec{q}_j}(\Delta),\\
(2)f_{\vec{q}_i}f_{\vec{q}_j}^{n-1}(H)&\cap f_{\vec{q}_j}f_{\vec{q}_i}^{n-1}(H)
=l(g_ig_j^{n-1}(H')\cap g_jg_i^{n-1}(H'))
=\emptyset.
\end{align*}
\end{proof}
\subsection{Definition of $K_\beta$ for $\beta\in(3/2,\beta^*]$} Because of the disappearance of the triple overlap, we only need one two-sided coin to be tossed to determine the transformation on the point. However, the partition of $S_\beta$ is complicated by the presence of holes.

First we partition $\Delta_1$ into $\tilde{E}_i$ and $\tilde{C}_{ij}$ (see Figure \ref{P2}), where
\begin{align*}
\tilde{E}_0&=([0,\frac{1}{\beta})\times [0,\frac{1}{\beta})) \backslash H,\\
\tilde{E}_1&=\{(x,y): 0 \leq y < \frac{1}{\beta},\frac{1}{\beta(\beta-1)} < x+y  \leq  \frac{1}{\beta-1}\} \backslash H,\\
\tilde{E}_2&=\{(x,y): 0 \leq x < \frac{1}{\beta},\frac{1}{\beta(\beta-1)}<x+y \leq \frac{1}{\beta-1}\} \backslash H,\\
\tilde{C}_{01}&=\{(x,y): x\geq \frac{1}{\beta}, y\geq  0 , x+y \leq \frac{1}{\beta(\beta-1)}\},\\
\tilde{C}_{12}&=\{(x,y): x\geq \frac{1}{\beta}, y\geq \frac{1}{\beta}, x+y \leq \frac{1}{\beta-1}\},\\
\tilde{C}_{02}&=\{(x,y): x\geq 0,  y\geq \frac{1}{\beta}, x+y \leq  \frac{1}{\beta(\beta-1)}\}.
\end{align*}
Here $H=\{(x,y): x<\frac{1}{\beta},y<\frac{1}{\beta},x+y>\frac{1}{\beta(\beta-1)}\}$.

Then we have the following lemma.
\begin{figure}[h]
\begin{tikzpicture}[scale=3.6,very thick]
\pgfsetfillopacity{0.6}
\fill[fill=blue!50,draw=black,thin] (0,0) -- (25/24,0)node[below]{$\frac{1}{\beta(\beta-1)}$}-- (0,25/24)node[left]{$\frac{1}{\beta(\beta-1)}$}--cycle;
\fill[fill=blue!50,draw=black,thin] (5/8,0)node[below]{$\frac{1}{\beta}$}--(5/3,0)node[below]{$\frac{1}{\beta-1}$}--(5/8,25/24)--cycle;
\fill[fill=blue!50,draw=black,thin] (0,5/8)node[left]{$\frac{1}{\beta}$} -- (25/24,5/8)-- (0,5/3)node[left]{$\frac{1}{\beta-1}$}--cycle;
\draw [->] (-0.3,0) -- (2,0) node[at end, below] {$x$};
\draw [->] (0,-0.15) -- (0,1.9) node[at end, left] {$y$};
\node[below] at (4/24,2/8) {$\tilde{E}_0$};
\node[below] at (3/4,2/8) {$\tilde{C}_{01}$};
\node[below] at (29/24,2/8) {$\tilde{E}_1$};
\node[below] at (4/24,20/24) {$\tilde{C}_{02}$};
\node[below] at (3/4,20/24) {$\tilde{C}_{12}$};
\node[below] at (4/24,29/24) {$\tilde{E}_2$};
\node[below] at (13/24,16/24) {$H$};
\end{tikzpicture}
\caption{$\Delta_1$ for $\frac{3}{2}<\beta\leq \beta^*$}
\label{P2}
\end{figure}
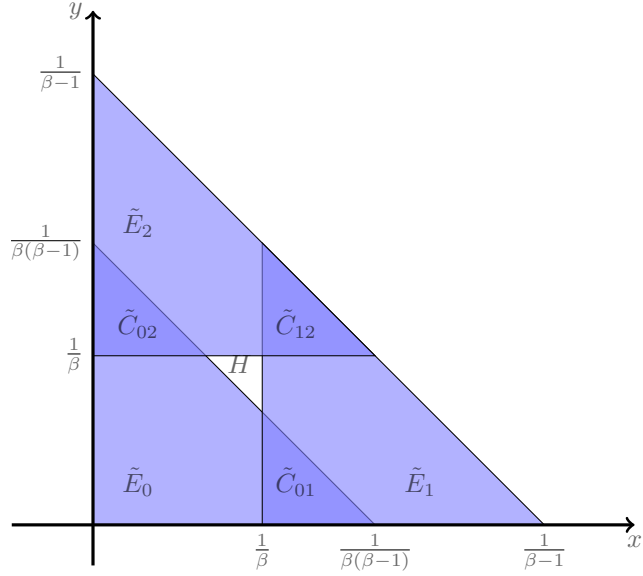

\begin{lemma}\label{delta2}
For $\beta\in(3/2,\beta^*]$.
Let $(x,y)\in S_\beta$ and $(x,y)=\sum_{i=1}^\infty\frac{a_i}{\beta^i}$ with $a_i\in\{\vec{q}_0,\vec{q}_1,\vec{q}_2\}$ be a representation of $(x,y)$ in base $\beta$. one has\\
(\romannumeral1) If $(x,y)\in \tilde{E}_i$ for some $i\in\{0,1,2\}$, then $a_1=\vec{q}_i$,\\
(\romannumeral2) If $(x,y)\in \tilde{C}_{ij}$ for some $ij\in\{01,12,02\}$, then $a_1\in\{\vec{q}_i,\vec{q}_{j}\}$.
\end{lemma}
\begin{proof} The proof is similar with Lemma \ref{delta1}.
\end{proof}

From the discussion in Proposition \ref{structure}, we know that all holes $f_{\vec{q}_i}^n(H)$ in the attractor $S_\beta$ are in $\tilde{E}_i$(see Figure \ref{radial}).
And the `holes' $f_{\vec{q}_i}f_{\vec{q}_j}^n(H)(i\neq j)$ that should exist in $\tilde{C}_i$ are covered by $f_{\vec{q}_j}(S_\beta)$, i.e. $f_{\vec{q}_i}f_{\vec{q}_j}^n(H)(i\neq j)\subset S_\beta$,
which leaves only one choice $\vec{q}_j$ for the first digit of the expansions of the points in $f_{\vec{q}_i}f_{\vec{q}_j}^n(H)$.
Now we can partition the Sierpinski carpet $S_\beta$ into equality regions $E_i$ and switch regions $C_{ij}$.
\begin{align*}
E_0&=\tilde{E}_0
\bigcup \cup_{n=1}^\infty f_{\vec{q}_1}f_{\vec{q}_0}^n(H)
\bigcup \cup_{n=1}^\infty f_{\vec{q}_2}f_{\vec{q}_0}^n(H)
\setminus\cup_{n=0}^\infty f_{\vec{q}_0}^nH,\\
E_1&=\tilde{E}_1
\bigcup \cup_{n=1}^\infty f_{\vec{q}_0}f_{\vec{q}_1}^n(H)
\bigcup \cup_{n=1}^\infty f_{\vec{q}_2}f_{\vec{q}_1}^n(H)
\setminus \cup_{n=0}^\infty f_{\vec{q}_1}^nH,\\
E_2&=\tilde{E}_2
\bigcup \cup_{n=1}^\infty f_{\vec{q}_1}f_{\vec{q}_2}^n(H)
\bigcup \cup_{n=1}^\infty f_{\vec{q}_0}f_{\vec{q}_2}^n(H)
\setminus \cup_{n=0}^\infty f_{\vec{q}_2}^nH,\\
C_{01}&=\tilde{C}_{01}\setminus  (\cup_{n=1}^\infty f_{\vec{q}_0}f_{\vec{q}_1}^n(H) \bigcup \cup_{n=1}^\infty f_{\vec{q}_1}f_{\vec{q}_0}^n(H)),\\
C_{12}&=\tilde{C}_{12}\setminus  (\cup_{n=1}^\infty f_{\vec{q}_2}f_{\vec{q}_1}^n(H) \bigcup \cup_{n=1}^\infty f_{\vec{q}_1}f_{\vec{q}_2}^n(H)),\\
C_{02}&=\tilde{C}_{02}\setminus  (\cup_{n=1}^\infty f_{\vec{q}_0}f_{\vec{q}_2}^n(H) \bigcup \cup_{n=1}^\infty f_{\vec{q}_2}f_{\vec{q}_0}^n(H)).
\end{align*}
The selection of the digits in the expansion becomes more accurate.
\begin{lemma}\label{lem w1}
For $\beta\in(3/2,\beta^*]$.
Let $\vec{z}\in S_\beta$ and $\vec{z}=\sum_{i=1}^\infty\frac{a_i}{\beta^i}$ with $a_i\in\{\vec{q}_0,\vec{q}_1,\vec{q}_2\}$ be a representation of $\vec{z}$ in base $\beta$. one has\\
(\romannumeral1) If $\vec{z}\in E_i$ for some $i\in\{0,1,2\}$, then $a_1=\vec{q}_i$,\\
(\romannumeral2) If $\vec{z}\in C_{ij}$ for some $ij\in\{01,12,02\}$, then $a_1\in\{\vec{q}_i,\vec{q}_{j}\}$.
\end{lemma}

\begin{proof}
It is enough to prove that $a_1=\vec{q}_j$ for $\vec{z}\in f_{\vec{q}_i}f_{\vec{q}_j}^n(H)$.

Suppose $a_1\neq \vec{q}_j$. Since $f_{\vec{q}_i}f_{\vec{q}_j}^n(H)\subset \tilde{C}_{ij}$, it follows from Lemma \ref{delta2} that $a_1=\vec{q}_i$. Then
$$f_{\vec{q}_i}^{-1}(\vec{z})= \beta\vec{z}-\vec{q}_i=\sum_{i=1}^\infty\frac{a_{i+1}}{\beta^i}\in S_\beta,$$
which is a contradiction to the fact that $f_{\vec{q}_i}^{-1}(\vec{z})\in f_{\vec{q}_j}^n(H)$ which is a hole.
\end{proof}
Let
$$C=C_{01}\cup C_{12}\cup C_{02},\ \ E=E_0 \cup E_1 \cup E_2.$$
Recall $ \Omega= \{0,1\}^\mathbb{N}$ with the product $\sigma$-algebra $\mathcal{A}$ and $\sigma:\Omega\rightarrow\Omega$ is the left shift,
and define $ K_\beta:\Omega \times S_\beta \rightarrow \Omega \times S_\beta$ by

\begin{equation*}
K_\beta(\omega,\vec{z})=
\begin{cases}
(\omega,\beta\vec{z}-\vec{q}_i), &\mbox{if $\vec{z}\in E_i,i=0,1,2$},\\
(\sigma\omega,\beta \vec{z}-\vec{q}_i), &\mbox{if $\omega_1=0$ and $\vec{z}\in C_{ij},\ ij\in\{01,12,02\}$,}\\
(\sigma\omega,\beta \vec{z}-\vec{q}_j), &\mbox{if $\omega_1=1$ and $\vec{z}\in C_{ij},\ ij\in\{01,12,02\}$.}
\end{cases}
\end{equation*}
The digits are given by
\begin{equation*}
d_1 = d_1(w,\vec{z})=
\begin{cases}
\vec{q}_i,  &\mbox{if $\vec{z}\in E_i,i=0,1,2$},\\
\  &\mbox{or $(\omega,\vec{z})\in \{\omega_1=0\}\times C_{ij},\ ij\in\{01,12,02\}$,}\\
\vec{q}_j, &\mbox{if $(\omega,\vec{z})\in \{\omega_1=1\}\times C_{ij},\ ij\in\{01,12,02\}$.}
\end{cases}
\end{equation*}
Then
\begin{equation*}
K_\beta(\omega,\vec{z})=
\begin{cases}
(\omega,\beta\vec{z}-d_1), &\mbox{if $\vec{z}\in E$},\\
(\sigma\omega,\beta \vec{z}-d_1), &\mbox{if $\vec{z}\in C$}.
\end{cases}
\end{equation*}

Set $d_n=d_n(\omega,\vec{z})=d_1(K_\beta^{n-1}(\omega,\vec{z}))$, and $\pi_2:\Omega \times S_\beta \rightarrow S_\beta$ be the canonical projection onto the second coordinate, then
$$\Vert\vec{z}-\sum_{i=1}^n\frac{d_i}{\beta^i}\Vert_1=\frac{\Vert\pi_2(K_\beta^{n}(\omega,\vec{z}))\Vert_1}{\beta_n}.$$
Since $\pi_2(K_\beta^{n}(w,\vec{z}))\in \Delta$ and $\Delta$ is a bounded set in $\mathbb{R}^2$, it follows that
 for all $\omega\in\Omega$ and for all $\vec{z}\in S_\beta$ one has
$$\vec{z}=\sum_{i=1}^\infty \frac{d_i}{\beta^i}= \sum_{i=1}^\infty \frac{d_i(\omega,\vec{z})}{\beta^i}.$$

One can see that the random map $K_\beta$ for $\beta\in(3/2,\beta^*]$ satisfies similar properties as that for  $\beta\in(1,3/2]$. For each point $\vec{z}\in S_\beta$, consider the set
$$D_{\vec{z}}=\{(d_1(\omega,\vec{z}),d_2(\omega,\vec{z}),\ldots):\omega\in\Omega\}.$$
\begin{theorem}
Suppose $\omega,\omega'\in\Omega$ are such that $\omega\prec\omega'$. Then
$$(d_1(\omega,\vec{z}),d_2(\omega,\vec{z}),\ldots)\preceq (d_1(\omega',\vec{z}),d_2(\omega',\vec{z}),\ldots).$$
\end{theorem}
\begin{proof}
This result is easier to prove than Theorem \ref{3.1} using the same method.
\end{proof}
\begin{theorem}\label{thm w1}
For $\beta\in(3/2,\beta^*]$.
Let $\vec{z}\in S_\beta$ and $\vec{z}=\sum_{i=1}^\infty\frac{a_i}{\beta^i}$ with $a_i\in\{\vec{q}_0,\vec{q}_1,\vec{q}_2\}$ be a representation of $\vec{z}$ in base $\beta.$ Then there exists an  $\omega\in\Omega$ such that $a_i=d_i(\omega,\vec{z})$.
\end{theorem}

\begin{proof}
Using Lemma \ref{lem w1}, the proof is similar with Theorem \ref{thm w}.
\end{proof}

\subsection{Unique $K_\beta$-invariant measure of maximal entropy.} On the set $\Omega\times S_\beta$ we consider the product $\sigma$-algebra $\mathcal{A\times S}$. Recall the function
$\rho_2: \{\vec{q}_0,\vec{q}_1,\vec{q}_2\}^\mathbb{N} \rightarrow \Upsilon$ which is given by
$$\rho_2(\vec{q}_{b_1},\vec{q}_{b_2},\vec{q}_{b_3},\ldots)=(b_1,b_2,b_3,\ldots).$$
Define the function
$\rho: \Omega \times S_\beta\rightarrow \{\vec{q}_0,\vec{q}_1,\vec{q}_2\}^\mathbb{N}$ by
$$\rho_1(\omega,\vec{z})=(d_1(\omega,\vec{z}), d_2(\omega,\vec{z}),\ldots).$$
Denote by $\varphi=\rho_2\circ\rho_1$ the function from $\Omega \times S_\beta$ to $\Upsilon$.
It is easy to check that $\varphi$ is measurable and $\varphi\circ K_\beta=\sigma'\circ \varphi $. It follows from Theorem \ref{thm w1} that $\varphi$ is surjective.

Let
\begin{align*}
Z&=\{(\omega,\vec{z})\in\Omega\times S_\beta: K_\beta^n(\omega,\vec{z})\in\Omega\times C \text{ infinitely often}\},\\
D&=\{(b_1,b_2,\ldots)\in \Upsilon: \sum_{i=1}^\infty\frac{\vec{q}_{b_{j+i-1}}}{\beta^i}\in C \text{ for infinitely many } j\text{'s}\}.
\end{align*}
Then $\varphi(Z)=D, K_\beta^{-1}(Z)=Z$ and $(\sigma')^{-1}(\Upsilon)=\Upsilon.$ Let $\varphi'$ be the restriction of the map $\varphi$ on $Z$.

\begin{lemma}\label{bi}
The map $\varphi':Z\rightarrow D$ is a bimeasurable bijection.
\end{lemma}
\begin{proof}
For any sequence $(b_1,b_2,\ldots)\in D$, we can get $\vec{z}=\sum_{i=1}^\infty\frac{\vec{q}_{b_i}}{\beta^i}$.
To get $\omega$, define
$$r_1=\min\{j\geq1:\sum_{i=1}^\infty\frac{\vec{q}_{b_{j+i-1}}}{\beta^i}\in C\},\ \
r_k=\min\{j>r_{k-1}:\sum_{i=1}^\infty\frac{\vec{q}_{b_{j+i-1}}}{\beta^i}\in C\}.$$
If $b_{r_k}\in\{0,2\}$, let $\omega_k=b_{r_k}/2$.
If $b_{r_k}=1$ and $\sum_{i=1}^\infty\frac{\vec{q}_{b_{r_k+i-1}}}{\beta^i}\in C_0$, let $\omega_k=1$.
If $b_{r_k}=1$ and $\sum_{i=1}^\infty\frac{\vec{q}_{b_{r_k+i-1}}}{\beta^i}\in C_1$, let $\omega_k=0$.
Define $(\varphi')^{-1}:\Upsilon'\rightarrow Z$ by
$$(\varphi')^{-1}((b_1,b_2,\ldots))=(w,\vec{z}).$$
$(\varphi')^{-1}$ is measurable by Theorem \ref{polish}, and is the inverse of $\varphi'$.
\end{proof}

\begin{lemma}\label{fullmeasure}
$\mathbb{P}(D)=1$, where $\mathbb{P}$ is the uniform product measure.
\end{lemma}
\begin{proof}
The proof is similar with Theorem \ref{m}. For any sequence $(b_1,b_2,\ldots)\in \Upsilon$ and $l\geq 3$, define
$$\vec{z}_l=(x_l,y_l)=\frac{\vec{q}_1}{\beta}+\frac{\vec{q}_2}{\beta^3}+\frac{\vec{q}_{b_1}}{\beta^{l+1}}+\frac{\vec{q}_{b_2}}{\beta^{l+2}}+\cdots.$$
Since $\beta\in[3/2,\beta^*)$ and $\lim_{n\rightarrow 0}\frac{1}{\beta^n(\beta-1)}=0$, there exsits $L>0$ such that for any $l\geq L$,
\begin{align*}\label{y1}
&x_l\geq \frac{1}{\beta},\ \ \ \ y_l\geq \frac{1}{\beta^3},\\
&x_l+y_l\leq\frac{1}{\beta}+\frac{1}{\beta^3}+\frac{1}{\beta^n(\beta-1)}\leq\frac{1}{\beta(\beta-1)}.
\end{align*}
It follows that $(x_l,y_l)\in \tilde{C}_{01}$. Notice that for any point $(x,y)\in \cup_{n=1}^\infty f_{\vec{q}_0}f_{\vec{q}_1}^n(H)\bigcup\cup_{n=1}^\infty f_{\vec{q}_1}f_{\vec{q}_0}^n(H)$, we have $y<1/\beta^3$, which implies that $(x_l,y_l)\in C_{01}\subset C$ for any $l\geq L$. Let
\begin{equation*}
D'=\{(b_1,b_2,\ldots)\in \Upsilon: b_jb_{j+1}\ldots b_{j+L-1}=1\underbrace{00\ldots 0}_{N-1 \text{ times }} \text{ for infinitely many } j\text{'s}\},
\end{equation*}
then $D'\subset D$. Since $\mathbb{P}(D')=1$, we have $\mathbb{P}(D)=1$.
\end{proof}

Now, consider the $K_\beta$-invariant measure
$\nu_\beta$ defined on
$\mathcal{A}\times \mathcal{S}$ by
$\nu_\beta (A)=\mathbb{P}(\varphi(Z \cap A))$. We have the following theorem.
\begin{theorem}
Let $3/2<\beta\leq\beta^*$. The dynamical systems
$(\Omega\times S_\beta, \mathcal{A}\times \mathcal{S},\nu_\beta, K_\beta)$
and
$(\Upsilon,\mathcal{B},\mathbb{P},\sigma')$
 are isomorphic. Moreover, the measure $\nu_\beta$ is the unique $K_\beta$-invariant measure of maximal entropy.
\end{theorem}

\begin{proof}
It follows from Lemmas \ref{bi} and \ref{fullmeasure} that $\varphi: (\Omega\times S_\beta, \mathcal{A}\times \mathcal{S},\nu_\beta, K_\beta)\rightarrow (\Upsilon,\mathcal{B},\mathbb{P},\sigma') $ is an isomorphism,
which implies that $h_{\nu_\beta}(K_\beta)=h_{\mathbb{P}}(\sigma')=\log 3$ and any other $K_\beta$-invariant measure with support $Z$ has entropy strictly less than $\log 3$.
Let $\mu$ be a  $K_\beta$-invariant measure for which $\mu(Z^c)>0$. By a similar discussion with that of Lemma \ref{max entropy}, we have $h_\mu(K_\beta) < \log 3.$
\end{proof}

\section*{Acknowledgements}
The first and third author were supported by National Natural Science Foundation of China (NSFC) No.~12071148 and Science and Technology Commission of Shanghai Municipality (STCSM)  No.~18dz2271000.
The first author was also supported by China Scholarship Council No.202006140156 during her visit to Utrecht University.

\end{document}